\documentclass[11pt]{amsart}

\usepackage[T1]{fontenc}
\usepackage{lmodern}
\usepackage{microtype}
\usepackage{amsmath,amssymb,esint}
\usepackage{amsrefs}
\usepackage[margin=1.08in]{geometry}
\usepackage{xcolor}
\usepackage[colorlinks=true,linkcolor=blue!50!black,citecolor=blue!50!black,urlcolor=blue!50!black]{hyperref}
\hypersetup{
  pdftitle={The weak-type (1,1) bound for the Hardy--Littlewood maximal function is $O(\sqrt{n} \log n)$},
  pdfauthor={Daniel Spector, Cody B. Stockdale},
  pdfsubject={The weak-type (1,1) bound for the Hardy--Littlewood maximal function is $O(\sqrt{n} \log n)$},
  hypertexnames=false
}

\newtheorem{theorem}{Theorem}[section]
\newtheorem{proposition}[theorem]{Proposition}
\newtheorem{lemma}[theorem]{Lemma}

\numberwithin{equation}{section}
\theoremstyle{remark}

\author{Daniel Spector}
\address{Daniel Spector, Department of Mathematics, National Taiwan Normal University, No. 88, Section 4, Tingzhou Road, Wenshan District, Taipei City, Taiwan 116, R.O.C.
\newline
Department of Mathematics, University of Pittsburgh, Pittsburgh, PA 15261, USA
}
\email{spectda@gapps.ntnu.edu.tw}
\thanks{D. Spector is supported by the National Science and Technology Council of Taiwan under research grant number 113-2115-M-003-017-MY3.}

\author{Cody B. Stockdale}
\address{Cody B. Stockdale, School of Mathematical and Statistical Sciences, Clemson University, Clemson, SC 29634, USA}
\email{cbstock@clemson.edu}
\thanks{C. B. Stockdale is supported by the National Science Foundation, DMS grant No. 2555710.}

\title[Weak-type bounds for maximal functions]
{The weak-type (1,1) bound for the Hardy--Littlewood maximal function is $O(\sqrt{n} \log n)$}
\begin{document}

\begin{abstract}
We prove a weak-type $(1,1)$ estimate for the centered Hardy--Littlewood maximal function with respect to Euclidean balls with dimensional dependence $O(\sqrt{n} \log n)$.  This improves the order of growth in the classical $O(n)$ estimate of Stein and Str\"omberg.  The proof goes through a pointwise bound of the Hardy--Littlewood maximal operator by the heat maximal operator with $\sqrt{n}$ loss.  The key technical aspect of our result is an improvement of the weak-type bound for the heat maximal operator from $O(\sqrt{n})$ to $O(\log n)$.  
\end{abstract}

\maketitle

\section{Introduction}
The centered Hardy--Littlewood maximal function over Euclidean balls of a function $f \in L^1_{\text{loc}}(\mathbb{R}^n)$ is defined as \looseness=-1
\begin{align*}
 \mathcal{M} f(x):=\sup_{r>0}\frac1{\omega_n r^n}
                   \int_{B(x,r)}|f(y)|\,d y,
\end{align*}
where  $\omega_n:=\pi^{n/2}/\Gamma(n/2+1)$ is the measure of the unit ball in $\mathbb{R}^n$.  A covering argument gives
\begin{align}\label{vitali}
 	\| \mathcal{M}f\|_{L^{1,\infty}(\mathbb{R}^n)}:= \sup_{\alpha > 0} \alpha |\{x\in\mathbb{R}^n\colon | \mathcal{M}f(x)|>\alpha\}| \leq 3^n \|f\|_{L^1(\mathbb{R}^n)}
\end{align}
for all $f \in L^1(\mathbb{R}^n)$, while the pioneering paper of Stein and Str\"omberg \cite[Theorem 3 on p.~264]{SteinStromberg} showed that \eqref{vitali} can be improved to a bound that is linear in the dimension: 
\begin{align}\label{hds_HL}
\| \mathcal{M}f\|_{L^{1,\infty}(\mathbb{R}^n)} \lesssim n \|f\|_{L^1(\mathbb{R}^n)}
\end{align}
for all $f \in L^1(\mathbb{R}^n)$. Here and in the sequel, we use the notation $A \lesssim B$ if there exists a dimension-free constant $C>0$ such that $A\leq CB$ and write $A \approx B$ if $A \lesssim B \lesssim A$. 

In the same paper, Stein and Str\"omberg asked whether the constant in the weak-type $(1,1)$ bound for $\mathcal{M}$ can be taken to be independent of dimension, a question mentioned again in Stein's conference proceedings for the 1986 International Congress of Mathematicians \cite[Problem (a) on p.~203]{Stein1987ICM}.  Despite the large body of research in the area \cite{Aldaz,DK,GG,LQ,MSW,NW,NiWr,Z,Aubrun2009,Bourgain1986AJM, Bourgain1986Israel, Bourgain2014Cube,BMSWCubes,BMSWBalls, BMSWVariational, BMSWSurvey, Carbery1986, DeleavalGuedonMaurey, Muller1990,KMPW2023, NaorTao2010, Stein1983BAMS}, which was catalyzed by  \cite{SteinStromberg} and the series of papers by Stein on infinite-dimensional harmonic analysis \cite{Stein1982BAMS,Stein1983BAMS,SteinStromberg}, to our knowledge, the order of growth in the general $O(n)$ Euclidean ball weak-type $(1,1)$ upper bound has not previously been
improved, and no lower bound diverging with the dimension is known.  We refer the reader to \cite{Aubrun2009,AldazCubes,IakovlevStromberg} for a negative answer to the analogous question of whether the cube maximal function admits a dimension-free bound. For balls restricted to radially decreasing functions, the optimal weak-type constant is $1$; see \cite{AldazPerezLazaro}.

Our main result is the following improved dimensional weak-type $(1,1)$ bound for $\mathcal{M}$. 
\begin{theorem}\label{cor:hl-main}
There is an absolute constant $C>0$ such that for every $n \in \mathbb{N}$ and every $f \in L^1(\mathbb{R}^n)$,
\[
\| \mathcal{M}f\|_{L^{1,\infty}(\mathbb{R}^n)} \leq C \sqrt n\log(1+n) \|f\|_{L^1(\mathbb{R}^n)}.
\]
\end{theorem}

To describe our proof, let us recall the approach of Stein and Str\"omberg in \cite{SteinStromberg}, which consists of two components and is in the spirit of the dimension-free estimates for the Poisson maximal operator in \cite[pp.~48--49]{SteinTopics}.  The first is the Hopf--Dunford--Schwartz maximal ergodic theorem \cite[Lemma VIII.7.6 and Theorem VIII.7.7, pp.~690--691]{DunfordSchwartz}, which gives a dimension-free bound for the maximal operator associated with the ergodic averages
\begin{align}\label{ergodic_maximal}
\tilde{P}_tf(x) := \frac{1}{t}\int_0^t H_sf(x)\;ds.
\end{align}
Here, $H_t$ denotes the heat propagator 
\[
 H_tf:=e^{t\Delta}f = h_t\ast f,
\]
where $h_t(x):=(4\pi t)^{-n/2}e^{-|x|^2/(4t)}$ is the heat kernel.  The second is a pointwise bound for the Hardy--Littlewood maximal function by the corresponding ergodic maximal operator:
\begin{align*}
\mathcal{M} f(x) \lesssim n \sup_{t>0} \tilde{P}_tf(x)
\end{align*}
for nonnegative $f$. The bound \eqref{hds_HL} follows from these steps, as they give 
$$
	\|\mathcal{M}f\|_{L^{1,\infty}(\mathbb{R}^n)} \lesssim n\Big\|\sup_{t>0} \tilde{P}_tf\Big\|_{L^{1,\infty}(\mathbb{R}^n)} \lesssim n\|f\|_{L^1(\mathbb{R}^n)}
$$
for all nonnegative $f \in L^1(\mathbb{R}^n)$, from which the general result follows.

Our first departure from the approach of Stein and Str\"omberg is to instead dominate the Hardy--Littlewood maximal function by the ``larger'' heat
maximal operator 
$$
 \mathcal{H}_{\ast}f(x):=\sup_{t>0}|H_tf(x)|,
$$
which only gives up a factor of $\sqrt{n}$:  If $f \in L^1_{\text{loc}}(\mathbb{R}^n)$ is nonnegative, then 
\begin{align}\label{prop:ball-heat-intro}
 \mathcal{M} f(x) \lesssim \sqrt{n} \mathcal{H}_{\ast} f(x)
\end{align}
for almost every $x \in \mathbb{R}^n$. To prove Theorem \ref{cor:hl-main}, it therefore suffices to establish a weak-type $(1,1)$ bound for $\mathcal{H}_*$ that grows logarithmically in the dimension. 
\begin{theorem}\label{thm:main}
There is an absolute constant $C>0$ such that for every $n \in \mathbb{N}$ and every $f \in L^1(\mathbb{R}^n)$,
\[
\|  \mathcal{H}_{\ast}f\|_{L^{1,\infty}(\mathbb{R}^n)} \leq C \log(1+n) \|f\|_{L^1(\mathbb{R}^n)}.
\]
\end{theorem}

Theorem \ref{thm:main} contains the main difficulty in our argument; we complete the introduction by outlining its proof.  It suffices to treat the case $n \geq 3$; see Section \ref{main}. By a standard approximation argument, it suffices to consider a finite positive atomic measure $\nu$ of mass $\mathfrak{m}$.  Following our previous work \cite{OuyangSpectorStockdale} on a dimension-free weak-type $(1,1)$ bound for the Riesz transform, we utilize a partial balayage approach.  While our earlier work used fractional Laplacian partial balayage, the present argument uses the classical Laplacian partial balayage corresponding to the second-order case.  We recall that the cap-$\kappa$ Laplacian partial balayage of $\nu$ has the form
\[
 \nu=\kappa\mathbf{1}_{\Omega}\,d x-\Delta u,
\]
where $u$ is nonnegative and $\Omega:=\{x \in \mathbb{R}^n\colon u(x)>0\}$ satisfies $|\Omega|=\frac{\mathfrak{m}}{\kappa}$. For convenience, we introduce the notation
\begin{align*}
\sigma:=\nu-\kappa\mathbf{1}_{\Omega}\,d x=-\Delta u.
\end{align*}
The term $\mu:=\kappa\mathbf{1}_{\Omega}\,d x$ is harmless; 
as we will see in the sequel, it is convenient to choose whether to estimate operators applied to $\nu$ or $\sigma$, depending on the regime.  To make this precise, for $t>0$, we further introduce the residual operator
\begin{align*}
 R_t\sigma:=H_t\sigma-\tilde{P}_t\sigma,
\end{align*}
where $\tilde{P}_t\sigma$ is the ergodic averaging operator defined in \eqref{ergodic_maximal}.  Then
\begin{align}\label{goodandbad}
H_t\nu = H_t\mu+  \tilde{P}_t\sigma + R_t \sigma
\end{align}
with $\mu := \kappa\mathbf{1}_{\Omega} \geq 0$ and
\begin{equation}
 \tilde{P}_t\sigma(x)=-\frac{H_tu(x)}t\le0
 \label{subordinated_sign}
\end{equation}
for $x \in {\Omega^c}$.
The identity \eqref{subordinated_sign} is a simple consequence of the relation $\partial_sH_su=-H_s\sigma$, the fundamental theorem of calculus, and the fact that $u(x)=0$ for $x \in {\Omega^c}$.

Our goal is to ultimately obtain a bound for $\mathcal{H}_{\ast}\nu$ in the space $L^{1,\infty}(\mathbb{R}^n)$.   To accomplish this, for a fixed $\alpha>0$, one introduces a balayage cap $\kappa=\kappa(\alpha)$ for which an application of the partial balayage yields a corresponding balayage domain $\Omega=\Omega(\alpha)$ and odometer $u$.  Division of the space into the balayage domain and its complement yields
\begin{align*}%\label{Omega_omega_c}
\alpha\,|\{ \mathcal{H}_{\ast}\nu>\alpha\}| = \alpha\,|\{x \in \Omega\colon \mathcal{H}_{\ast}\nu(x)>\alpha\}|  + \alpha\,|\{x \in \Omega^c\colon \mathcal{H}_{\ast}\nu(x)>\alpha\}|.
\end{align*}
For the first term on the right-hand side above, there is nothing to estimate, as the size bound for $|\Omega|$ obtained in the balayage is sufficient.  For the second term, we use the fact that for $x \in \Omega^c$, the relation \eqref{goodandbad} and the sign of $\tilde{P}_t\sigma$ yield the upper bound
\begin{align*}
 \mathcal{H}_{\ast}\nu(x) \le \mathcal{H}_{\ast}\mu(x) + \sup_{t>0}|R_t \sigma(x)|.
\end{align*}
With $\kappa = \kappa(\alpha)$ suitably chosen, the fact that $\mathcal{H}_{\ast}\mu \leq \kappa$ allows one to reduce the question to a weak-type $(1,1)$ estimate for the residual maximal operator:  It remains to prove that
\begin{align}\label{residual}
\alpha\,\Big|\Big\{x \in \Omega^c\colon \sup_{t>0}|R_t \sigma(x)|>\alpha\Big\}\Big| \lesssim  \log(1+n) \mathfrak m.
\end{align}

Estimates for the left-hand side of \eqref{residual} seem to be non-trivial, and the argument becomes more involved at this point.  For $x \in {\Omega^c}$, consider the logarithmic-time reparametrizations 
\begin{align*}
 R_x(v):=R_{e^v}\sigma(x) \qquad \text{and}\qquad %,\nonumber\\
 F_{\eta,x}(v):=H_{e^v}\eta(x),  %\label{log_sigma}
\end{align*}
where $\eta = \sigma, \mu, \text{ or } \nu$ and note that
\[
\sup_{t>0}|R_t \sigma(x)| =  \sup_{v \in \mathbb{R}} |R_x(v)|.
\]
Our choice to place the $x$ variable as a subscript stems from the usefulness of the Fourier transform in logarithmic time in establishing our estimates, for which we use the convention
\begin{align*}
 \widehat  F_{\eta,x}(\xi)&=\int_{\mathbb{R}}e^{-2\pi i\xi v} F_{\eta,x}(v)\,d v
\end{align*}
and
\begin{align*}
  F_{\eta,x}(v)&=\int_{\mathbb R}e^{2\pi i\xi v}\widehat  F_{\eta,x}(\xi)\,d\xi
\end{align*}
for the inverse Fourier transform, whenever inversion is justified.

The logarithmic change of variables and Fourier transform is equivalent to taking the Mellin transform, though we find the former viewpoint more comfortable.  In particular, the Fourier transform in logarithmic time yields several important identities. First, as justified in the proof of Theorem~\ref{thm:main}, the
logarithmic-time orbits satisfy
 \begin{equation*}
 \widehat R_x(\xi)=r(\xi)\widehat F_{\sigma, x}(\xi)
\end{equation*}
for almost every $x\in\Omega^c$, where $r(\xi)=\frac{2\pi i\xi}{1+2\pi i\xi}$. Second, if one defines the Abel mean associated with $-\Delta$ by
 \begin{align}\label{Abel}
 A_s:=s(s-\Delta)^{-1}
     =\int_0^\infty se^{-st}H_t\,d t
\end{align}
for $s>0$, then the Fourier transform of the logarithmic time heat evolution of $\nu$ is related to the square function of $A_s^N:=\underbrace{A_s\circ A_s \circ \cdots \circ A_s}_{N \text{ times}}$, the Abel power of order $N$ in $L^2(\Omega^c)$:
\begin{align}\label{Abel-identity}
\int_{\Omega^c} \int_{\mathbb{R}}
 (2\pi\xi)^2 |M_N(\xi)|^2 |\widehat F_{\nu, x}(\xi)|^2\,d\xi\,d x =  \int_{\Omega^c} \int_0^\infty
 |s\partial_s(A_s^N\nu)(x)|^2\frac{d s}{s}\,d x,
 \end{align}
where 
\begin{align*}
 |M_N(\xi)|^2 :=  \frac{|\Gamma(N+2\pi i\xi)|^2}{\Gamma(N)^2}.
\end{align*}
These facts play an important role in our estimates for the high and low Mellin frequencies 
\begin{align}
 Q_\Lambda(x)&:=\int_{|\xi|>\Lambda} |\widehat F_{\sigma,x}(\xi)|\,d\xi \label{high}
\end{align}
and
\begin{align}
 T_\Lambda(x)&:=\int_{|\xi|\le\Lambda}  |r(\xi)\widehat F_{\sigma,x}(\xi)|\,d\xi,\label{low}
\end{align}
which, as established in the proof of Theorem~\ref{thm:main}, dominate the residual via the pointwise estimate
\begin{align}
\sup_{v\in\mathbb R}|R_x(v)|
\leq T_\Lambda(x)+Q_\Lambda(x)
\label{R_bound}
\end{align}
for almost every $x\in\Omega^c$.

We control the terms involving \eqref{high} and \eqref{low} by different arguments. 
The bound for \eqref{high} is somewhat simpler, as the Mellin formula recorded in Lemma \ref{lem:mellin} -- which is an exact computation for the Fourier transform of $F_{\sigma,x}$ %\eqref{log_sigma} 
-- and the application of several sophisticated identities for the Euler Gamma function yield the estimate
\begin{equation}
 \int_{\Omega^c} Q_\Lambda(x)\,d x= \int_{\Omega^c}\int_{|\xi|>\Lambda} |\widehat F_{\sigma,x}(\xi)|\,d\xi  \,d x  \lesssim C^n \Lambda^2e^{-c\Lambda} \mathfrak m
 \label{high-part}
\end{equation}
for universal constants $c,C>0$ and $\Lambda>n$.  To kill the exponential growth of $C^n$, one takes $\Lambda =An$ for a suitable choice of $A>1$.

As might be expected, one must pay for this choice of $\Lambda$ in the estimate for \eqref{low}.  Here, we go through an $L^2(\Omega^c)$ estimate and show that
\begin{equation}
 \int_{\Omega^c} T_\Lambda(x)^2\,d x
 \lesssim \log^2(e\Lambda)\,\kappa\mathfrak m.
 \label{remainder-max}
\end{equation}
This step is where the bound loses a logarithm of the dimension.  The inequality \eqref{remainder-max} is argued in several steps.  First, we establish an estimate for the square function of the Abel power of order $N$ for $\nu$ that appears on the right-hand side of \eqref{Abel-identity}.  This estimate bounds the squared $L^2$ norm of this square function by a quantity that is linear in $\kappa$ and has
square-root dependence in $N$.  This is the most technically challenging part of the paper, and, as it may be of independent interest, we record it here.\looseness=-1
\begin{theorem}
\label{thm:resolvent-square}
If $n\ge3$ and $\nu=\sum_{i=1}^M m_i\delta_{x_i}$ is a finite positive atomic measure of mass $\mathfrak{m}$ with partial balayage at height $\kappa>0$ given by
\[
 \nu=\kappa\mathbf{1}_\Omega\,d x-\Delta u,
\]
where $u$ is nonnegative and $\Omega=\{x \in \mathbb{R}^n\colon u(x)>0\}$, then 
\begin{equation}\label{eq:exterior-abel-square}
 \int_{\Omega^c} \int_0^\infty  |s\partial_s(A_s^N\nu)(x)|^2\frac{d s}{s}\,dx \lesssim \sqrt N\,\kappa\mathfrak m
\end{equation}
for all $N \in \mathbb{N}$.
\end{theorem}
\noindent
We remark that the appearance of $\kappa$ on the right-hand side is through its relation to $\Omega^c=\Omega^c(\kappa)$.

Second, a combination of Theorem \ref{thm:resolvent-square}, the identity \eqref{Abel-identity}, and the fact that
\[
 |M_N(\xi)|\gtrsim 1
\]
when $|\xi|\lesssim \sqrt{N}$ yield an estimate for the low Mellin frequencies of $\widehat F_{\nu,x}(\xi)$:  
\begin{equation}
 \int_{\Omega^c} \int_{|\xi|\le B}
 (2\pi\xi)^2|\widehat F_{\nu,x}(\xi)|^2\,d\xi\,d x
 \lesssim B\,\kappa\mathfrak{m}
 \label{lowMellin_nu}
\end{equation}
for $B\ge 1$. This can be transferred to $\widehat F_{\sigma,x}(\xi)$ because of harmless estimates for $\mu$ (or in this case, a similar expression to \eqref{lowMellin_nu} for $\widehat F_{\mu,x}(\xi)$).  Finally, summation over a multiple of $\log(\Lambda)$ dyadic scales yields the logarithmic loss in dimension asserted in \eqref{remainder-max}.  As above, the estimates which underlie \eqref{high-part} and \eqref{remainder-max} may be of independent interest, and we therefore summarize them in the following theorem expressed in terms of the logarithmic time orbits of $\sigma$.
\begin{theorem}
\label{thm:exterior}
Let $n\ge3$ and $\nu=\sum_{i=1}^M m_i\delta_{x_i}$ be a finite positive atomic measure of mass $\mathfrak{m}$ with partial balayage at height $\kappa>0$ given by
\[
 \nu=\kappa\mathbf{1}_\Omega\,d x-\Delta u,
\]
where $u$ is nonnegative and $\Omega=\{x \in \mathbb{R}^n\colon u(x)>0\}$. If $\sigma=\nu-\kappa\mathbf{1}_\Omega\,d x=-\Delta u$ and $F_{\sigma,x}(v)=H_{e^v}\sigma(x)$ for $x\in {\Omega^c}$, then 
\begin{equation}
 \int_{\Omega^c}\int_{|\xi|>An}
 |\widehat F_{\sigma,x}(\xi)|\,d\xi\,d x
 \lesssim \mathfrak{m}
 \label{eq:frequency-tail-intro}
\end{equation}
for any $A>1$ sufficiently large, and
\begin{equation}
 \int_{\Omega^c} \int_{|\xi|\le B}
 (2\pi\xi)^2|\widehat F_{\sigma,x}(\xi)|^2\,d\xi\,d x
 \lesssim B\,\kappa\mathfrak{m}
 \label{eq:frequency-prefix-intro}
\end{equation}
for every $B \geq 1$.
\end{theorem}

We do not know whether it is possible to remove the logarithmic dependence on dimension in our Theorem \ref{thm:main}.  In any case, a dimension-free bound for the heat maximal operator would give a bound $\sqrt n$ for the Hardy--Littlewood maximal operator, and this is the best achievable by our method because of the optimality\footnote{One can verify this by taking an approximation of the identity, or by working
directly with measures and taking $\nu=\delta_0$ in the inequality.} of the transfer in \eqref{prop:ball-heat-intro}.  

The plan of the paper is as follows.   In Section \ref{sec:reduction}, we recall the requisite details concerning the Laplacian partial balayage of a purely atomic measure and establish several facts concerning the outputs of the partial balayage theorem.  In Section \ref{sec:resolvent}, we develop the necessary machinery and give the proof of Theorem \ref{thm:resolvent-square}.  In Section \ref{sec:mellin}, we gather the results developed in the preceding sections to give a proof of Theorem \ref{thm:exterior}.  In Section \ref{lowfreq_log}, we derive \eqref{remainder-max} as a direct consequence of Theorem \ref{thm:exterior}.  We conclude the paper in Section \ref{main} by proving Theorems \ref{cor:hl-main} and \ref{thm:main}.

\section{Atomic reduction and Laplacian partial balayage}
\label{sec:reduction}

It is well known that to establish a weak-type $(1,1)$ estimate it suffices to prove the bound for \emph{finite positive atomic measures}, i.e., a finite sum 
\[
 \nu=\sum_{i=1}^M m_i\delta_{x_i}
\]
with $m_i >0$; see, e.g. \cite[Theorem 1]{MenarguezSoria1992}.  Note that the supports of the measures we consider are finite, hence contained in a compact set.

In this paper, we use the following form of partial balayage.  In the classical setting, partial balayage is a second-order obstacle problem associated with a cap constraint.  We include the global Euclidean construction, using only the least superharmonic majorant
formulation and local obstacle regularity; see \cite{KinderlehrerStampacchia,PetrosyanShahgholianUraltseva} for those standard facts and \cite{GustafssonRoos} for partial-balayage terminology. The statement below is a special case of the structure theorem for partial balayage in Euclidean space, \cite[Theorem~2.1 and (2.34), (2.40)--(2.41)]{GustafssonSakai} (with $R=\mathbb{R}^n$ and density $p\equiv\kappa$ in the notation there; the remark (2.41) covers measures that are singular with respect to Lebesgue measure, hence atomic $\nu$); see also \cite{GardinerSjodin}.  We include a proof for the convenience of the reader. \looseness=-1
\begin{proposition}
\label{prop:balayage}
If $n\ge3$, $\nu=\sum_{i=1}^M m_i\delta_{x_i}$ is a finite positive atomic measure of mass $\mathfrak{m}$, and $\kappa>0$, then there exist a nonnegative odometer $u$ with compact support and a bounded open set $\Omega=\{u>0\}$ containing every source atom in its interior with $|\Omega|=\frac{\mathfrak{m}}{\kappa}$ such that
\[
 \nu=\kappa\mathbf{1}_\Omega\,d x-\Delta u.
\]
In particular, $u\in L^1(\mathbb{R}^n)$, its canonical potential representative vanishes on $\Omega^c$, and is continuous in a neighborhood of $\Omega^c$.  Further, 
\begin{align*}
 \delta_\nu:=\operatorname{dist}(\operatorname{supp}\nu,\Omega^c)>0,   
\end{align*}
and
\begin{align*}
 \Delta u=\kappa\mathbf{1}_\Omega-\nu\le\kappa                
\end{align*}
in the sense of distributions.
\end{proposition}

\begin{proof}
By combining atoms at the same point, we may assume
that the points $x_i$ are distinct.  Let $\Gamma$ be the Newtonian kernel, normalized by $-\Delta\Gamma=\delta_0$, and write $U^\lambda=\Gamma*\lambda$.  Set
\[
 \gamma(x)=-\frac{\kappa}{2n}|x|^2-U^\nu(x),
\]
so that
\[
\Delta\gamma=\nu-\kappa
 \]
in the sense of distributions.  The function $\gamma$ is an obstacle, and our first goal is to show the existence of its least superharmonic majorant $s$.  In particular, the function $u$ claimed in the proposition is simply the difference between the least superharmonic majorant $s$ and $\gamma$, with its useful properties inherited from the minimality of $s$.  However, the existence of such a least superharmonic majorant is nontrivial and we proceed step by step.  

We begin by showing that the set of superharmonic majorants is non-empty.  To this end, for $1\le i\le M$, choose $r_i>0$ so that the ball $B(x_i,r_i)$,
with density $\kappa/M$, carries mass $m_i$:
\begin{align*}
r_i^n=\frac{Mm_i}{\kappa\omega_n}.
\end{align*}
The corresponding one-source odometer is
\begin{align*}
w_i
=U^{m_i\delta_{x_i}-(\kappa/M)
\mathbf1_{B(x_i,r_i)}\,dx}.
\end{align*}
By Newton's theorem \cite[Theorem~9.7]{LiebLoss},
$w_i$ is nonnegative and vanishes outside $B(x_i,r_i)$.
Moreover,
\begin{align*}
\Delta w_i
=\frac{\kappa}{M}\mathbf1_{B(x_i,r_i)}\,dx
-m_i\delta_{x_i}
\end{align*}
in the sense of distributions.  The sum of these odometers and $\gamma$ (as distributions) is 
\begin{align*}
s_0(x)
:=-\frac{\kappa}{2n}|x|^2
-\frac{\kappa}{M}\sum_{i=1}^M
U^{\mathbf1_{B(x_i,r_i)}\,dx}(x).
\end{align*}
We claim $s_0$ is a continuous superharmonic majorant of $\gamma$.  First, the definition of $s_0$ shows that it is a finite continuous function on all of $\mathbb R^n$.  Next, nonnegativity of $w_i$ implies 
\begin{align*}
s_0=\gamma+\sum_{i=1}^Mw_i\ge\gamma,
\end{align*}
away from the atoms, and of course at the atoms $s_0> \gamma = -\infty$, so that $s_0$ majorizes $\gamma$. Finally, the fact that
\begin{align*}
 \Delta s_0=-\kappa+\frac\kappa M
                   \sum_{i=1}^M\mathbf{1}_{B(x_i,r_i)}\le0
\end{align*}
in the sense of distributions shows that $s_0$ is superharmonic.  This completes the demonstration of the claim that $s_0$ is
a continuous superharmonic majorant of $\gamma$.  

In order to apply \cite[Theorem~7.4]{HKM}, which asserts that the lower semicontinuous regularization of the pointwise infimum of the superharmonic majorants of $\gamma$ is itself superharmonic, we need to verify that this family is locally uniformly bounded below.  This is trivial away from atoms, since there $\gamma$ is continuous, and therefore it remains to argue near atoms.  Fix $x_i$ and choose $r>0$ so that $\overline{B(x_i,r)}$ contains no other atom.
For every superharmonic majorant $v\ge\gamma$, the
minimum principle gives
\[
\inf_{B(x_i,r)}v
\ge \min_{\partial B(x_i,r)}\gamma>-\infty,
\]
and the lower bound is independent of $v$.  This proves the local uniform lower bound near each atom.  With these facts in hand, we define $h$ to be the pointwise infimum of the superharmonic majorants
of $\gamma$, and further let $s$ be its lower semicontinuous
regularization.  By \cite[Theorem~7.4 on p.~132]{HKM}, $s$ is superharmonic, while continuity of $\gamma$  away from the atoms and the fact that $\gamma=-\infty$ at the atoms shows that $s$ still majorizes $\gamma$. In particular, $h\le s$ by the definition of $h$, while the reverse inequality holds
by regularization, and therefore $s=h$. Thus, we have established the existence of a lower semicontinuous function $s$ which is
the least superharmonic majorant, and satisfies $\gamma\le s\le s_0$.

We next prove the Lewy--Stampacchia inequalities
\begin{equation}\label{eq:lewy-stampacchia}
-\kappa\le\Delta s\le0
\end{equation}
in $\mathcal D'(D)$, where
$D=\mathbb R^n\setminus\operatorname{supp}\nu$.
The upper bound follows from superharmonicity.
For the lower bound, set
\begin{align*}
S(x)=s(x)+\frac{\kappa}{2n}|x|^2.
\end{align*}
We want to show that $S$ is subharmonic.
The advantage of adding the quadratic term is that
the obstacle becomes $-U^\nu$, which is harmonic in $D$.

Fix a ball $B\Subset D$. Since $\gamma\le s\le s_0$
and both $\gamma$ and $s_0$ are continuous on
$\overline B$, the functions $s$ and $S$ are bounded
there. Let $H_BS$ be the Poisson extension of
$S|_{\partial B}$. The inequality $s\ge\gamma$ gives
$S\ge-U^\nu$. Since $-U^\nu$ is harmonic near
$\overline B$, comparison with the boundary values
of $H_BS$ gives
\begin{align*}
H_BS\ge-U^\nu
\end{align*}
in $B$. It follows that
\[
\widetilde s(x) =H_BS(x)-\frac{\kappa}{2n}|x|^2
\]
satisfies
\[
\widetilde s\geq\gamma
\qquad \text{and}\qquad
\Delta\widetilde s=-\kappa\leq0
\]
in $B$. Moreover, since $S$ is lower semicontinuous,
\[
\liminf_{\substack{y\to x\\y\in B}}\widetilde s(y)
\geq s(x)
\]
for every $x \in \partial B$.  The pasting lemma for superharmonic functions thus shows that the function
\[
w=
\begin{cases}
\min\{s,\widetilde s\} &\text{in }B\\
s &\text{in }\mathbb R^n\setminus B
\end{cases},
\]
is a superharmonic majorant of $\gamma$.  By the minimality of $s$, we have $s\leq w$.  Since $w\leq s$ by construction, it follows that $w=s$, and therefore
\[
s\leq\widetilde s
\]
in $B$. Equivalently,
\[
S\leq H_BS
\]
in $B$. 

Evaluating the inequality $S\leq H_BS$ at the center of $B$, we find
that the locally bounded function $S$ satisfies
\[
S(x)\leq\fint_{\partial B(x,r)}S(y)\,d\sigma(y)
\]
whenever $\overline{B(x,r)}\subset D$, i.e. that $\Delta S\ge0$
in $\mathcal D'(D)$.  In combination with the definition of $S$, we deduce that $\Delta s\ge-\kappa$, which completes the proof of \eqref{eq:lewy-stampacchia}. 

We now set $u=s-\gamma$. The bounds $\gamma\le s\le s_0$
give
\begin{align*}
0\le u\le s_0-\gamma=\sum_{i=1}^M w_i
\end{align*}
away from the atoms. Since the $w_i$ are compactly
supported, so is $u$.

To identify the Laplacian of $u$, we first observe that
$\Delta\gamma=-\kappa$ in $D$, so the
Lewy--Stampacchia inequalities give
\begin{align*}
0\le\Delta u=\Delta s+\kappa\le\kappa
\end{align*}
in $\mathcal D'(D)$. The bound $0\le\Delta u\le\kappa$ and interior
Calder\'on--Zygmund estimates imply that
$u\in W^{2,p}_{\mathrm{loc}}(D)$ for every $1<p<\infty$.

Set $\Omega=\{u>0\}$. Minimality of $s$ makes it
harmonic on $\Omega\cap D$, and therefore
$\Delta u=\kappa$ in this set. On the zero set of $u$, by two applications of Stampacchia's lemma \cite[Corollary~1.21]{HKM}, first to $u$ and then to the components of $\nabla u$, on deduces that $\nabla u$ and $D^2u$ vanish almost everywhere.  Therefore
\begin{align*}
\Delta u=\kappa\mathbf1_\Omega
\end{align*}
almost everywhere in $D$.  Thus, $u$ is a nonnegative solution of the classical obstacle problem in $D$ with constant right-hand side $\kappa$.  The optimal regularity theorem for the obstacle problem \cite[Theorem~2.3]{PetrosyanShahgholianUraltseva} (see also \cite{Caffarelli1998}) yields 
\[
u\in C^{1,1}_{\mathrm{loc}}(D).
\]
In particular, away from atoms, the distributional Laplacian of $u$ is the locally bounded function $\kappa\mathbf1_\Omega$ and has no
singular component supported on the free boundary.

It remains to analyze $u$ near the atoms.  Fix $x_i$ and choose $r>0$ so that
$\overline{B(x_i,r)}$ contains no other atom.
The lower bound established above, together with
$s\le s_0$ and continuity of $s_0$, shows that $s$
is bounded on $B(x_i,r)$.  On $B(x_i,r)$, write
\begin{align*}
u&=s-\gamma=m_i\Gamma(\,\cdot-x_i)+h_i,\\
h_i&:=s+\frac{\kappa}{2n}|\cdot|^2
+\sum_{j\ne i}m_j\Gamma(\,\cdot-x_j).
\end{align*}
Since $h_i$ is bounded and $\Gamma(x-x_i)$ tends to
infinity as $x$ approaches $x_i$, we can choose $r$
small enough that $u>0$ throughout $B(x_i,r)$.
The function $s$ is then harmonic on the punctured
ball. Its boundedness makes the singularity at $x_i$
removable, so $s$ is harmonic on the whole ball and
\begin{align*}
\Delta u=-\Delta\gamma=\kappa-m_i\delta_{x_i}.
\end{align*}

Repeating this argument at each atom and using
$\Delta u=\kappa\mathbf1_\Omega$ in $D$, we obtain
\begin{align*}
\Delta u=\kappa\mathbf1_\Omega\,dx-\nu
\end{align*}
as distributions on $\mathbb R^n$.
We have also shown that every source point has
a neighborhood contained in $\Omega$, and moreover, as there are
finitely many source points, we have shown that $\operatorname{supp}\nu\Subset\Omega$.

To compute $|\Omega|$, choose
$\chi\in C_c^\infty(\mathbb R^n)$ equal to one near
$\operatorname{supp}u$. Testing the identity above
against $\chi$ gives
\begin{align*}
\kappa|\Omega|=\nu(\mathbb R^n)=\mathfrak m.
\end{align*}
The local $C^{1,1}$ regularity away from the atoms
also shows that $u$ is continuous near $\Omega^c$.
\end{proof}

We will need an estimate for the average of $u$
on spheres centered outside $\Omega$.

\begin{lemma}
\label{lem:moment}
For $\nu=\sum_{i=1}^M m_i\delta_{x_i}$, let $u$ and $\Omega$ be the odometer and domain corresponding to an application of Proposition \ref{prop:balayage} at height $\kappa$.  For every $z\in {\Omega^c}$ and $r>0$,
\begin{align}
 \fint_{\partial B(z,r)}u
 \le\frac{\kappa r^2}{2n},                            \label{eq:2.5}
\end{align}
and therefore
\begin{align}
 \int_{B(z,r)}u(y)\,d y
 \le\frac{\kappa\omega_n}{2(n+2)}r^{n+2}.       \label{eq:2.6}
\end{align}
\end{lemma}

\begin{proof}
The canonical representative of
\[
 w(y)=u(y)-\frac{\kappa}{2n}|y-z|^2
\]
is superharmonic and satisfies $w(z)=0$.  Its spherical means do not exceed
its value at the center, which gives \eqref{eq:2.5}.  Integrating the
spherical means in $r$ gives \eqref{eq:2.6}.
\end{proof}

\begin{proposition}\label{finite_P}
Let $n\ge3$, let $\nu$ be a finite positive atomic measure of mass
$\mathfrak{m},$ and $\kappa>0$.  For $u$ the associated odometer given by Proposition \ref{prop:balayage}, we define $P\colon \Omega^c \to [0,+\infty]$ by
\begin{align}\label{P_estimate}
P(x) = \int_\Omega\frac{u(y)}{|x-y|^{n+2}}\,dy.
\end{align} 
Then there exists $C>0$ independent of $n$ such that
\begin{align*}
 \int_{\Omega^c}P(x)\,d x \leq \frac{n\omega_n}2 C^n\mathfrak{m},
\end{align*} 
and in particular, $P(x) < +\infty$ for almost every $x \in {\Omega^c}$.
\end{proposition}
\begin{proof}
Let $d(y)=\operatorname{dist}(y,\Omega^c)$ for $y\in\Omega$ and put
\[
 W(y)=\frac{u(y)}{d(y)^2}.
\]
Let $\mathcal W$ be the collection of maximal dyadic cubes $Q$ satisfying
$\operatorname{dist}(Q,\Omega^c)\ge8\operatorname{diam} Q$, a variant of the Whitney decomposition \cite[Chapter~VI, Theorem~1]{SteinSI}.  These cubes have pairwise disjoint interiors
and cover $\Omega$ up to their boundaries.  Indeed, for any fixed point of $\Omega$ one can find a sufficiently small dyadic cube which satisfies the inequality.
If $\widehat Q$ is the dyadic
parent of a maximal cube $Q$, then, with $s_Q=\operatorname{diam} Q$ and
$\delta_Q=\operatorname{dist}(Q,\Omega^c)$,
\[
 \operatorname{dist}(\widehat Q,\Omega^c)<8\operatorname{diam}\widehat Q=16s_Q.
\]
Since every point of $\widehat Q$ is within $2s_Q$ of $Q$, it follows that
\[
 8s_Q\le\delta_Q<18s_Q.                          
\]
For each $Q$, choose $z_Q\in {\Omega^c}$ and $y_Q\in\overline Q$ such that
$|z_Q-y_Q|=\delta_Q$.
Since $d\ge\delta_Q$ on $Q$, $Q\subset B(z_Q,\delta_Q+s_Q)$, and $u\ge0$,
Lemma~\ref{lem:moment} gives
\begin{align*}
 \int_QW
 &\le\frac{1}{\delta_Q^{2}}\int_Q u
 \le\frac{1}{\delta_Q^{2}}\int_{B(z_Q,\delta_Q+s_Q)}u
 \le\frac{\kappa\omega_n}{2(n+2)}
   \left(\frac{\delta_Q+s_Q}{\delta_Q}\right)^{n+2}\delta_Q^{\,n}\\
 &\le\frac{\kappa\omega_n}{2(n+2)}\left(\frac98\right)^{n+2}(18s_Q)^n
 =A_n\kappa|Q|,
\end{align*}
where
\[
 A_n:=\frac{\omega_n}{2(n+2)}
 \left(\frac98\right)^{n+2}(18\sqrt n)^n
 \le C^n.                                        
\]
Here we used $|Q|=(s_Q/\sqrt n)^n$ and the bound $\omega_n n^{n/2}\le C^n$,
which follows from $\omega_n=\pi^{n/2}/\Gamma(n/2+1)$ and Stirling's
formula \cite[Appendix~A.6]{Grafakos}.
Summing over $Q$ and using
$\kappa|\Omega|=\mathfrak{m}$ gives
\begin{equation}
\int_{\Omega} W(y)\, dy \le C^n\mathfrak{m}.                  \label{eq:4.4}
\end{equation}

We next estimate the integral of $P$ defined in \eqref{P_estimate}.  Since $|x-y|\ge d(y)$ for $x\in {\Omega^c}$ and $y\in\Omega$, Tonelli's theorem,
polar integration, and \eqref{eq:4.4} give
\begin{align*}
 \int_{\Omega^c}P(x)\,d x
 &\le\int_\Omega u(y)
       \int_{|x-y|\ge d(y)}|x-y|^{-n-2}\,d x\,d y\\
 &=\frac{n\omega_n}2\int_\Omega\frac{u(y)}{d(y)^2}\,d y
 \le\frac{n\omega_n}2 C^n\mathfrak{m} <+\infty.
\end{align*}
In particular, $P(x)<\infty$ for almost every $x\in {\Omega^c}$.
\end{proof}

\section{A Square Function Estimate for the Abel-powers}
\label{sec:resolvent}
In this section, we develop the necessary machinery and establish the proof of Theorem \ref{thm:resolvent-square}.  To this end, we recall the free Abel mean defined in \eqref{Abel}, for $s>0$, 
\[
 A_s:=s(s-\Delta)^{-1}
     =\int_0^\infty se^{-st}H_t\,d t. 
\]
The operator $A_s$ is positive and contractive on both $L^2(\mathbb R^n)$ and $L^\infty(\mathbb R^n)$, and satisfies the identity
\begin{equation}\label{eq:abel-derivative}
 s\partial_sA_s=A_s(I-A_s).
\end{equation}

In this section and the next, we identify the measure $\mu=\kappa\mathbf 1_\Omega\,d x$ with its density $\kappa\mathbf 1_\Omega$.  In particular, $\|\mu\|_\infty=\kappa$ and
$\|\mu\|_2^2=\kappa^2|\Omega|=\kappa\mathfrak m$.
It is convenient to argue the case $\nu\in C_c^\infty(\Omega)$ and $u\in H^1_0(\Omega)\cap L^\infty(\Omega)$ and then pass the limit.  To this end we first introduce smooth regularizations of the atomic measure $\nu$.

\subsection{Regularization of the atoms}

\begin{lemma}
\label{lem:radial-regularization}
There are nonnegative functions $\nu_\varepsilon\in C_c^\infty(\Omega)$
with $\int\nu_\varepsilon=\mathfrak m$ and
$\nu_\varepsilon\,d x\to\nu$ weakly, and bounded nonnegative functions
$u_\varepsilon\in H^1_0(\Omega)$ such that
\[
 \nu_\varepsilon=\mu-\Delta u_\varepsilon,
 \qquad \{u_\varepsilon>0\}=\Omega,
\]
and $u_\varepsilon=u$ in a fixed neighborhood of $\Omega^c$.
\end{lemma}

\begin{proof}
By combining atoms at the same point, we may assume that the points
$x_i$ are distinct.  Since the finitely many points lie in the open set $\Omega$, there exists $\varepsilon_0>0$ such that the closed balls $\overline{B(x_i,\varepsilon_0)}$
are pairwise disjoint and contained in $\Omega$.  Fix a radial nonnegative function $\eta \in C^\infty_c(B(0,1))$ with $\int\eta=1$ and by translation and dilation, for $0<\varepsilon\le\varepsilon_0$, define $\eta_{i,\varepsilon}\in C_c^\infty(B(x_i,\varepsilon))$ to be nonnegative, radial about $x_i$, with $\int\eta_{i,\varepsilon}=1$. Further set
\begin{align*}
 \nu_\varepsilon:=\sum_{i=1}^M m_i\eta_{i,\varepsilon} \qquad\text{and}\qquad  u_\varepsilon:=u-\sum_{i=1}^M m_i\bigl(\Gamma(\,\cdot-x_i)-\Gamma*\eta_{i,\varepsilon}\bigr),
\end{align*}
where the formula for $u_\varepsilon$ is initially understood away from the source atoms and its values at those points are defined by continuous extension, as justified
below.  Then $\nu_\varepsilon\in C_c^\infty(\Omega)$ is nonnegative and satisfies
$\int\nu_\varepsilon=\mathfrak m$, $\nu_\varepsilon\,dx\to\nu$ weakly as
$\varepsilon$ tends to zero.  Moreover, since $-\Delta\Gamma=\delta_0$,
\[
 -\Delta u_\varepsilon
 =-\Delta u-\sum_{i=1}^M m_i\bigl(\delta_{x_i}-\eta_{i,\varepsilon}\bigr)
 =(\nu-\mu)-\nu+\nu_\varepsilon
 =\nu_\varepsilon-\mu .
\]

We next prove the assertions that $u_\varepsilon=u$ on the fixed neighborhood $\mathbb R^n\setminus\bigcup_iB(x_i,\varepsilon_0)$ of $\Omega^c$, that
$u_\varepsilon$ is compactly supported, and $\{u_\varepsilon>0\}$ coincides with $\Omega$ off the balls $B(x_i,\varepsilon)$. Indeed, these follow from the fact that $\eta_{i,\varepsilon}$ is radial with total mass one and Newton's theorem \cite[Theorem~9.7]{LiebLoss}, which asserts that 
\[
\Gamma*\eta_{i,\varepsilon}=\Gamma(\,\cdot-x_i)
\]
 on $B(x_i,\varepsilon)^c$. In particular,
$
u_\varepsilon=u
$
on $\mathbb R^n\setminus\bigcup_iB(x_i,\varepsilon)$, from which the claims follow.

We next show that $u_\varepsilon>0$ on each ball $B(x_i,\varepsilon)$.  Together with the previous step, this gives $u_\varepsilon\ge0$ and $\{u_\varepsilon>0\}=\Omega$.  Fix a ball $B(x_i, \varepsilon)$.  By the proof of Proposition~\ref{prop:balayage}, 
\[
u=m_i\Gamma(\,\cdot-x_i)+h_i
\]
on $B(x_i,\varepsilon_0)$ with $h_i$ bounded, and the terms with $j\ne i$
in the definition of $u_\varepsilon$ vanish on $B(x_i,\varepsilon)$.
Therefore
\[
 u_\varepsilon=h_i+m_i\,\Gamma*\eta_{i,\varepsilon}
\]
on $B(x_i,\varepsilon)$. By Newton's theorem again, the radial potential
$\Gamma*\eta_{i,\varepsilon}$ is bounded below on $B(x_i,\varepsilon)$
by its value on $\partial B(x_i,\varepsilon)$, namely
$\Gamma(\varepsilon e_1)$, which tends to infinity as
$\varepsilon\downarrow0$ because $n\ge3$.  Hence
\begin{align*}
u_\varepsilon\ge m_i\Gamma(\varepsilon e_1)-\sup|h_i|>0 
 \end{align*}
on $B(x_i,\varepsilon)$ for all sufficiently small $\varepsilon$.  Since
$\Gamma*\eta_{i,\varepsilon}$ is also bounded, $u_\varepsilon$ is
bounded, and altogether $u_\varepsilon\ge0$ with
$\{u_\varepsilon>0\}=\Omega$.

Finally, we argue that $u_\varepsilon \in H^1_0(\Omega)$.  That $u_\varepsilon\in
L^2(\mathbb{R}^n)$ follows from the fact that $u_\varepsilon$ is bounded with compact support.  Concerning higher regularity, the formula 
\begin{align*}
\Delta u_\varepsilon=\mu-\nu_\varepsilon\in
L^2(\mathbb{R}^n)
\end{align*}
yields $u_\varepsilon\in H^2(\mathbb{R}^n)$ by Plancherel.
Near $\Omega^c$ we have $u_\varepsilon=u$, and therefore the local $C^{1,1}$ regularity of Proposition~\ref{prop:balayage} together with $u=\nabla u=0$ on
$\Omega^c$ gives that 
\[
 u_\varepsilon=O(d^2) \qquad \text{and}\qquad  |\nabla u_\varepsilon|=O(d),
\]
where we recall the notation $d(x)=\operatorname{dist}(x,\Omega^c)$.  Thus, one can use regularized-distance cutoffs \cite[Chapter~VI, Section~2, Theorem~2]{SteinSI} to approximate $u_\varepsilon$ in $H^1$ by functions compactly supported in $\Omega$, proving the claim.
\end{proof}

All estimates below are uniform in $\varepsilon$, and for clarity of presentation we suppress this notation.

\subsection{Free and killed Abel means}

Put $V=\kappa/u$ on $\Omega$ and let $K$ be the nonnegative
self-adjoint operator associated with the closed form
\[
 \mathfrak q[f]
 =\int_\Omega|\nabla f|^2\,d x+
   \int_\Omega V|f|^2\,d x
\]
for $f \in \mathcal{D}(\frak q) = H^1_0(\Omega)\cap L^2(\Omega,V\,d x)$.  The fact that the regularized odometer $u$ is continuous and strictly positive on $\Omega$ implies that the potential $V=\kappa/u$ is nonnegative, measurable, and locally bounded.  Moreover, the multiplication form
\begin{align*}
f\mapsto\int_\Omega V|f|^2\,dx
\end{align*}
with domain $L^2(\Omega)\cap L^2(\Omega,V\,dx)$ is a closed nonnegative form on $L^2(\Omega)$.   Adding the Dirichlet form gives a closed form on
$H^1_0(\Omega)\cap L^2(\Omega,V\,dx)$.
This domain contains $C_c^\infty(\Omega)$ and is
therefore dense in $L^2(\Omega)$. The associated
operator is $K=-\Delta_\Omega+V$, with the sum
understood in the form sense.

We next compute the action of $K$ on $u$.
We already know that $u\in H^1_0(\Omega)$, and
\begin{align*}
\int_\Omega Vu^2\,dx=\kappa\int_\Omega u\,dx<\infty,
\end{align*}
so $u$ belongs to the form domain.
For $\varphi\in C_c^\infty(\Omega)$, the identity
$-\Delta u=\nu-\mu$ from
Lemma~\ref{lem:radial-regularization} gives
\begin{align*}
\mathfrak q(u,\varphi)
&=\int_\Omega(\nu-\mu)\varphi\,dx
  +\int_\Omega Vu\,\varphi\,dx\\
&=\int_\Omega\nu\varphi\,dx,
\end{align*}
because $Vu=\mu$.
Since $\nabla u$, $Vu=\kappa\mathbf 1_\Omega$, and $\nu$ belong to $L^2$,
each term is continuous in $\varphi$ for the $H^1$ norm, and
$C_c^\infty(\Omega)$ is dense in $H^1_0(\Omega)\supset\mathcal D(\mathfrak q)$.
Hence 
\[\mathfrak q(u,\varphi)=\int_\Omega\nu\varphi\,dx
\]
for all $\varphi\in\mathcal D(\mathfrak q)$, which, by the definition of the operator associated with a closed form \cite[Section~1.2.3, Definition~1.21]{Ouhabaz}, means that $u\in\mathcal D(K)$, 
\begin{equation}\label{eq:Ku}
 Ku=\nu, \qquad \text{and} \qquad
 \qquad K^{-1}\nu=u.
\end{equation}
Here $K^{-1}$ is well defined because $\Omega$ is bounded and Poincar\'e's inequality \cite[Section~5.6.1, Theorem~3]{Evans} implies $K$ has a positive spectral gap. Define the killed Abel mean $B_s:=s(s+K)^{-1}$ and, for $j\geq0$, set
\begin{align*}
 z_j(s):=s^j(s+K)^{-j-1}\nu \qquad \text{and}\qquad  \rho_j(s):=Vz_j(s).
\end{align*}
We call the functions $\rho_j(s)$ the \emph{absorption layers}.

\begin{lemma}
For $j\geq0$ and $s>0$,
\begin{align*}
 0\leq z_j(s), \rho_j(s) \leq u \qquad\text{and}\qquad  0\leq\rho_j(s)\leq\mu.
\end{align*}
After extension by zero from $\Omega$ to $\mathbb R^n$,
\begin{align}
 (s-\Delta)z_0(s)&=\nu-\rho_0(s),
 \label{eq:layer-equation-zero}\\
 (s-\Delta)z_j(s)&=s z_{j-1}(s)-\rho_j(s),
 \qquad j\geq1.
 \label{eq:layer-equation}
\end{align}
Therefore, for every $N\geq1$,
\begin{equation}\label{eq:free-killed-comparison}
 (A_s^N-B_s^N)\nu
 =\sum_{j=0}^{N-1}A_s^{N-j}\rho_j(s).
\end{equation}
In particular,
\begin{equation}\label{eq:exterior-comparison}
 \mathbf1_{\Omega^c}A_s^N\nu
 =\mathbf1_{\Omega^c}(A_s^N-B_s^N)\nu.
\end{equation}
\end{lemma}

\begin{proof}
By \eqref{eq:Ku},
\[
z_0(s)=(s+K)^{-1}Ku=u-B_su.
\]
We first prove that the resolvent $(s+K)^{-1}$ is positive.
Let $f\in L^2(\Omega)$ be nonnegative and set
\[
w=(s+K)^{-1}f.
\]
Then
\[
s\int_\Omega w\varphi\,dx+\mathfrak q(w,\varphi)
=\int_\Omega f\varphi\,dx
\]
for every $\varphi\in\mathcal D(\mathfrak q)$. Noting that  $w^-=\max\{-w,0\} \in \mathcal D(\mathfrak q)$, taking $\varphi=w^-$ gives
\[
\begin{aligned}
0
&\leq\int_\Omega fw^-\,dx=s\int_\Omega ww^-\,dx+\mathfrak q(w,w^-)\\
&=-s\|w^-\|_2^2-\|\nabla w^-\|_2^2
  -\int_\Omega V|w^-|^2\,dx
\leq0.
\end{aligned}
\]
It follows that $w^-=0$, and hence $w\geq0$. Thus $(s+K)^{-1}$ is positive and, by linearity, order preserving. An application of this observation to the nonnegative functions $\nu$ and $u$ gives
\begin{align*}
z_0(s)=(s+K)^{-1}\nu\geq0 \qquad\text{and}\qquad B_su=s(s+K)^{-1}u\geq0.
\end{align*}
The identity $z_0(s)=u-B_su$ therefore yields
\begin{align*}
0&\leq z_0(s), B_su\leq u.
%0&\leq B_su=u-z_0(s)\leq u.
\end{align*}
The positivity and order preserving properties of $B_s$ then give 
\begin{align*}
0\leq z_j(s)=B_s^jz_0(s)
\leq B_s^ju\leq u
\end{align*}
by induction. As a result,
\begin{align*}
0\leq\rho_j(s)=Vz_j(s)\leq Vu=\mu.
\end{align*}

We next justify that the extension of $z_j(s)$ by zero from $\Omega$ to $\mathbb R^n$ satisfies the corresponding distributional equation on $\mathbb R^n$ without producing a boundary measure.  The constants in this boundary argument may depend on $n$, $\nu$, $\kappa$, and the test function.  They are used only to justify the limiting argument and do not enter the final estimates.

The local $C^{1,1}$ regularity
from Proposition~\ref{prop:balayage}, together with
$u=\nabla u=0$ on $\Omega^c$, gives
\begin{align*}
u(x)\le C_{n,\nu,\kappa}d(x)^2
\end{align*}
near $\Omega^c$. The same constant works for all
the regularizations, since they agree with the
original odometer in a fixed neighborhood of
$\Omega^c$.

Choose a cutoff $\chi_\delta\in C_c^\infty(\Omega)$
equal to one where $d(x)>2\delta$, with
\begin{align*}
|\nabla\chi_\delta(x)|
&\le C_{n,\nu,\kappa}\delta^{-1},\\
|\nabla^2\chi_\delta(x)|
&\le C_{n,\nu,\kappa}\delta^{-2}.
\end{align*}
Set
\begin{align*}
E_\delta=\{x\in\Omega:0<d(x)<2\delta\}.
\end{align*}
Fix $\varphi\in C_c^\infty(\mathbb R^n)$. To obtain the distributional equation for the zero extension of $z_j(s)$, we test the interior equation with $\chi_\delta\varphi$ and let $\delta$ tend to zero. Away from the boundary, $\chi_\delta$ converges to one, so it remains only to show that the terms supported in $E_\delta$ tend to zero.

The first such term satisfies
\[
\begin{aligned}
\limsup_{\delta \to 0} \left|
\int_{E_\delta}(1-\chi_\delta)
\nabla z_j(s)\cdot\nabla\varphi\,dx
\right|
\leq
 \limsup_{\delta \to 0} \|\nabla\varphi\|_\infty
|E_\delta|^{1/2}
\|\nabla z_j(s)\|_{L^2(E_\delta)}=0
\end{aligned}
\]
since $|E_\delta| \to 0$ and $\nabla z_j(s)\in L^2(\Omega)$. We must also control the term produced when the derivative falls on the cutoff. Integration by parts gives
\begin{align*}
\int_\Omega
\varphi\nabla z_j(s)\cdot\nabla\chi_\delta\,dx
=-\int_\Omega
z_j(s)\nabla\varphi\cdot\nabla\chi_\delta\,dx-\int_\Omega
z_j(s)\varphi\Delta\chi_\delta\,dx,
\end{align*}
where both integrals on the right-hand side are in fact supported in $E_\delta$. On this set,
\[
0\leq z_j(s)\leq u
\leq C_{n,\nu,\kappa} d(x)^2
\lesssim \delta^2,
\]
which together with the bounds for $|\nabla\chi_\delta|$ and $|\nabla^2 \chi_\delta|$  yield
\[
\limsup_{\delta \to 0} \left|
\int_\Omega
z_j(s)\nabla\varphi\cdot\nabla\chi_\delta\,dx
\right|
\leq C_{n,\nu,\kappa}  \limsup_{\delta \to 0} \delta|E_\delta|
=0
\]
and
\[
\limsup_{\delta \to 0} \left|
\int_\Omega
z_j(s)\varphi\Delta\chi_\delta\,dx
\right|
\leq C_{n,\nu,\kappa}  \limsup_{\delta \to 0}|E_\delta|
=0.
\]

The remaining zero-order terms converge by dominated convergence. In particular, the potential term is controlled by
\[
0\leq \rho_j(s)\leq\kappa\mathbf1_\Omega.
\]
Thus we are justified to pass to the limit in the weak formulation, and in particular this establishes the identities \eqref{eq:layer-equation-zero} and
\eqref{eq:layer-equation} for the zero extensions of $z_j(s)$.

Applying $(s-\Delta)^{-1}$ and iterating gives
\[
 A_s^N\nu
 =s z_{N-1}(s)+
   \sum_{j=0}^{N-1}A_s^{N-j}\rho_j(s).
\]
The preceding identity and the relation $s z_{N-1}(s)=B_s^N\nu$ proves
\eqref{eq:free-killed-comparison}, while \eqref{eq:exterior-comparison} follows from the fact that the zero extension of $B_s^N\nu$ vanishes almost everywhere in $\Omega^c$.
\end{proof}

\subsection{The absorption layers}

All properties of the killed evolution used below are collected in the
following lemma.
\begin{lemma}
For $s>0$,
\begin{align}
 \sum_{j=0}^\infty\rho_j(s)&=\mu,
\nonumber\\
 0\leq\rho_j(s)&\lesssim \frac{1}{\sqrt{j+1}}\,\mu \quad \text{ for all }  j\geq0.
 \label{eq:layer-pointwise}
\end{align}
For $j\geq1$,
\begin{equation}\label{eq:layer-scale-mass}
 \int_0^\infty\|\rho_j(s)\|_1\frac{d s}{s}
 =\frac{\mathfrak m}{j}.
\end{equation}
For every $j\geq0$, the map $s \mapsto \rho_j(s)$ is locally absolutely continuous as an $L^2(\Omega)$-valued map and
\begin{equation}\label{eq:layer-derivative}
 s\partial_s\rho_j(s)
 =j\rho_j(s)-(j+1)\rho_{j+1}(s).
\end{equation}
For $1\leq j<N$,
\begin{equation}\label{eq:beta-poisson}
 j\rho_j(s)
 =N\int_0^1\rho_N(s/\theta)
   \frac{(N-1)!\,\theta^{j-1}(1-\theta)^{N-j-1}}
        {(j-1)!(N-j-1)!}\,d\theta.
\end{equation}
Finally,
\begin{align}
\label{eq:layer-endpoints} \lim_{s \to 0^+} \rho_0(s)=\mu,\\
 \label{eq:layer-endpoints2}  \lim_{s \to \infty}\rho_0(s)=0
\end{align}
in $L^2(\Omega)$.
\end{lemma}

\begin{proof}
Define $p_j(r)=e^{-r}r^j/j!$ and let
\[
g_t=e^{-tK}\nu.
\]
We claim that $g_t$ is nonnegative. Indeed, positivity
of the resolvent established above implies that
\[
(I+tK/m)^{-1}
=\frac{m}{t}\left(\frac{m}{t}+K\right)^{-1}
\]
preserves nonnegativity for every $m\ge1$. By the
spectral theorem,
\[
e^{-tK}\nu
=\lim_{m\to\infty}(I+tK/m)^{-m}\nu
\]
in $L^2(\Omega)$. Since $\nu\ge0$ and the nonnegative cone is closed
in $L^2(\Omega)$, we conclude that $g_t\ge0$.
The Gamma formula for resolvent powers \cite[Proposition~11.1, (11.5)]{Komatsu} gives
\begin{align*}
 z_j(s)=\int_0^\infty p_j(st)g_t\,d t \qquad\text{and}\qquad
u= K^{-1}\nu= \int_0^\infty g_t\,d t.
\end{align*}
Thus
\[
 \rho_j(s)=V\int_0^\infty p_j(st)g_t\,d t.
\]
These facts and the identity $\sum_{j=0}^\infty p_j(r)=1$
yield the claim
\begin{equation*}
\sum_{j=0}^\infty \rho_j(s)=Vu=\mu.
\end{equation*} 
For $j\geq1$, by Stirling's formula
\cite[Appendix~A.6]{Grafakos}, we have the bound
\[
\sup_{r>0}p_j(r)=p_j(j)=e^{-j}\frac{j^j}{j!}
\lesssim \frac{1}{\sqrt{j+1}}.
\]
For $j=0$, the same bound follows directly from
$p_0(r)=e^{-r}\leq1$. This proves the case $j=0$ recorded in \eqref{eq:layer-pointwise}, while for $j\geq1$, the identity
\[
 \int_0^\infty p_j(st)\frac{d s}{s}=\frac1j,
\]
Tonelli's theorem, and 
\[
\int_\Omega Vu=\mathfrak m
\]
give \eqref{eq:layer-scale-mass}.

The scalar identity
\[
 r \partial_r p_j(r)=jp_j(r)-(j+1)p_{j+1}(r)
\]
proves \eqref{eq:layer-derivative}.  For $s$ in a compact subinterval of $(0,\infty)$ one has
\[
|\partial_s p_j(st)|=|t\, (\partial_r p_j)(st)|\le (2j+1)/s
\]
uniformly in $t$.  Moreover, this identity and the bound $0\le p_k\le1$ show that, on every compact subinterval of $(0,\infty)$, a constant multiple of
\[
V\int_0^\infty g_t\,dt=\mu\in L^2(\mathbb R^n)
\]
dominates the differentiated integral.

The Beta--Poisson identity
\[
 jp_j(r)
 =N\int_0^1p_N(r/\theta)
   \frac{(N-1)!\,\theta^{j-1}(1-\theta)^{N-j-1}}
        {(j-1)!(N-j-1)!}\,d\theta
\]
follows from the substitution $y=r(1-\theta)/\theta$, while integration
against $Vg_t\,d t$ gives \eqref{eq:beta-poisson}.

Finally,
\[
 \rho_0(s)=V\int_0^\infty e^{-st}g_t\,d t.
\]
This converges pointwise to $Vu=\mu$ as $s$ tends to zero and to zero as
$s$ tends to infinity.  Since $0\leq\rho_0(s)\leq\mu\in L^2(\mathbb{R}^n)$, dominated
convergence proves \eqref{eq:layer-endpoints} and \eqref{eq:layer-endpoints2}.
\end{proof}

Taking the derivative of \eqref{eq:free-killed-comparison}, using
\eqref{eq:abel-derivative} and \eqref{eq:layer-derivative}, and
telescoping the sum gives
\begin{equation}\label{eq:free-killed-derivative}
 s\partial_s\bigl[(A_s^N-B_s^N)\nu\bigr]
 =N\sum_{j=0}^{N-1}A_s^{N-j}(I-A_s)\rho_j(s)
  -NA_s\rho_N(s).
\end{equation}

\subsection{Almost orthogonality of the layers}

\begin{theorem}
\label{thm:free-killed-square}
For every integer $N\geq1$ and every regularized $\nu$ as above,
\begin{equation}\label{eq:free-killed-square}
 \int_0^\infty
 \left\|s\partial_s\bigl[(A_s^N-B_s^N)\nu\bigr]\right\|_2^2
 \frac{d s}{s}
 \leq C\sqrt N\,\kappa\mathfrak m.
\end{equation}
\end{theorem}

\begin{proof}
The inequalities \eqref{eq:layer-pointwise}, \eqref{eq:layer-scale-mass}, and the fact that $\|\mu\|_\infty=\kappa$ imply, for $j\geq1$,
\begin{equation}\label{eq:layer-energy}
 \mathcal B_j
 :=\int_0^\infty\|\rho_j(s)\|_2^2\frac{d s}{s}
 \lesssim\kappa\mathfrak m\,j^{-3/2}.
\end{equation}

For $0\leq a\leq1$ set $q_\ell(a)=Na^\ell(1-a)$.  We estimate the various terms on the right-hand side of \eqref{eq:free-killed-derivative}.

To estimate the first summand, set
$E_N(s)=\langle\rho_0(s),A_s^N\rho_0(s)\rangle$.  Then
\[
 s\partial_s E_N(s)
 =N\langle\rho_0,A_s^N(I-A_s)\rho_0\rangle
  -2\langle\rho_1,A_s^N\rho_0\rangle.
\]
The endpoint values vanish by \eqref{eq:layer-endpoints}, \eqref{eq:layer-endpoints2}, and the strong
convergence $A_s^N\to0$ on $L^2$ as $s$ tends to zero ($A_s^N$ converges to the projection onto $\ker(-\Delta)$, which is trivial in $L^2(\mathbb{R}^n)$).  By the bounds
$0\leq A_s^N\rho_0\leq\kappa$, integration with respect to the Haar measure $d s/s$ gives
\[
 \int_0^\infty
 \langle\rho_0,A_s^N(I-A_s)\rho_0\rangle\frac{d s}{s}
 \leq\frac{2\kappa\mathfrak m}{N}.
\]
Since $0\leq q_N\leq1$, spectral calculus gives $q_N^2\leq q_N$.
Therefore
\begin{equation}\label{eq:zeroth-layer}
 \int_0^\infty\|q_N(A_s)\rho_0(s)\|_2^2\frac{d s}{s}
 \leq2\kappa\mathfrak m.
\end{equation}

Put $J=\lceil N/2\rceil$.  If $1\leq j<J$, then
\[
 \|q_{N-j}(A_s)\|_{2\to2}
 \leq N\max_{0\leq a\leq1}a^{N-j}(1-a)
 \leq\frac{N}{N-j+1}\leq2.
\]
Minkowski's inequality and \eqref{eq:layer-energy} imply
\begin{equation}\label{eq:early-layers}
 \begin{split}
 \int_0^\infty
 \left\|\sum_{j=1}^{J-1}q_{N-j}(A_s)\rho_j(s)\right\|_2^2
 \frac{d s}{s}
 \lesssim\kappa\mathfrak m
 \left(\sum_{j<J}j^{-3/4}\right)^2\lesssim\sqrt N\,\kappa\mathfrak m.
 \end{split}
\end{equation}

To estimate the remaining summands, set
\[
 L_N(s)=\sum_{j=J}^{N-1}q_{N-j}(A_s)\rho_j(s).
\]
After taking the spatial Fourier transform, fix
$\zeta\in\mathbb R^n$, put $\ell=4\pi^2|\zeta|^2$, and write
$a_s=s/(s+\ell)$.  The Beta identity and the change of variables
$t=s/\theta$ give
\[
 \widehat{L_N(s)}(\zeta)
 =\int_s^\infty K_{N,\ell}(s,t)
       \bigl[N\widehat{\rho_N(t)}(\zeta)\bigr]\frac{d t}{t},
\]
where
\[
 K_{N,\ell}(s,t)
 =Na_s(1-a_s)
 \sum_{j=J}^{N-1}\binom{N-1}{j}
 \left(\frac{s}{t}\right)^j
 \left[a_s\left(1-\frac{s}{t}\right)\right]^{N-1-j}
 \mathbf1_{\{t>s\}}.
\]
The kernel $K_{N,\ell}$ is nonnegative. Integrating
its summands using the Beta integral gives
\begin{align*}
\int_s^\infty K_{N,\ell}(s,t)\,\frac{dt}{t}
&=N(1-a_s)\sum_{j=J}^{N-1}\frac{a_s^{N-j}}{j}\\
&\le\frac{N}{J}\le2.
\end{align*}

For the column estimate, suppose $\ell>0$ and fix $t$.
With $a_t=t/(t+\ell)$, we have
\begin{align*}
\frac{s}{t}+a_s\left(1-\frac{s}{t}\right)
&=\frac{a_s}{a_t},\\
a_s(1-a_s)\frac{ds}{s}&=da_s.
\end{align*}
Adding the omitted terms of the binomial sum and
changing variables from $s$ to $a_s$ therefore gives
\begin{align*}
\int_0^t K_{N,\ell}(s,t)\,\frac{ds}{s}
&\le N\int_0^{a_t}
\left(\frac{a}{a_t}\right)^{N-1}\,da\\
&=a_t\le1.
\end{align*}
When $\ell=0$, the kernel vanishes, so both estimates
hold in that case as well.  Schur's test at each
spatial frequency, followed by Plancherel and Tonelli, gives
\begin{equation}\label{eq:late-layers}
 \int_0^\infty\|L_N(s)\|_2^2\frac{d s}{s}
 \leq2N^2\mathcal B_N
 \lesssim\sqrt N\,\kappa\mathfrak m.
\end{equation}
For the final term, contractivity of $A_s$ and
\eqref{eq:layer-energy} give
\begin{equation}\label{eq:terminal-layer}
\int_0^\infty\|NA_s\rho_N(s)\|_2^2\,\frac{ds}{s}
\le N^2\mathcal B_N
\lesssim\sqrt N\,\kappa\mathfrak m.
\end{equation}
The triangle inequality in $L^2(dx\,ds/s)$, applied
to \eqref{eq:free-killed-derivative}, now proves
\eqref{eq:free-killed-square}.
\end{proof}

\begin{proof}[Proof of Theorem~\ref{thm:resolvent-square}]
We apply Theorem~\ref{thm:free-killed-square} to the
radial regularizations $\nu_\varepsilon$ constructed
in Lemma~\ref{lem:radial-regularization}.  By \eqref{eq:exterior-comparison},
\[
 \int_{\Omega^c}\int_0^\infty
 \left|s\partial_s(A_s^N\nu_\varepsilon)(x)\right|^2
 \frac{d s}{s}\,d x
 \lesssim\sqrt N\,\kappa\mathfrak m,
\]
uniformly in $\varepsilon$.  

For fixed $s>0$ and $x\in \Omega^c$, the free
resolvent kernel and its $s$-derivative are continuous on the common
compact support of the regularizing measures.  Hence
\[
 s\partial_s(A_s^N\nu_\varepsilon)(x)
 \longrightarrow s\partial_s(A_s^N\nu)(x),
\]
and \eqref{eq:exterior-abel-square} follows from Fatou's lemma.
\end{proof}

\section{The High and Low Mellin Frequencies}\label{sec:mellin}
In this section we give the proof of Theorem \ref{thm:exterior}.  We first recall the Mellin formula, which will be useful in estimating the high Mellin frequencies.  This is the content of the next lemma which is derived directly from $H_t\sigma=-\Delta H_tu$.  

\begin{lemma}
\label{lem:mellin}
For almost every $x\in {\Omega^c}$ and every $\xi\in\mathbb{R}$,
\begin{align}
 \widehat F_{\sigma,x}(\xi)
 &=-\pi^{-n/2}4^{2\pi i\xi}\Gamma\!\left(\frac n2+2\pi i\xi\right)
 (n+4\pi i\xi)(2+4\pi i\xi)                         \notag\\
 &\qquad\times
 \int_\Omega u(y)|x-y|^{-n-2-4\pi i\xi}\,d y.      \label{eq:5.1}
\end{align}
In particular,
\begin{align}
 |\widehat F_{\sigma,x}(\xi)|
 \le\pi^{-n/2}\left|\Gamma\!\left(\frac n2+2\pi i\xi\right)\right|
 |(n+4\pi i\xi)(2+4\pi i\xi)|P(x).               \label{eq:5.2}
\end{align}
\end{lemma}

\begin{proof}
By Proposition \ref{finite_P}, $P(x)<\infty$ for almost every $x\in {\Omega^c}$.  For such
$x$,
\begin{equation*}
 \int_\Omega u(y)\int_0^\infty
 |\Delta h_t(x-y)|\frac{\,d t}{t}\,d y
 \le C_nP(x)<\infty,                            
\end{equation*}
and we may therefore interchange the order of integration by Fubini's theorem.  The substitution
$s=|x-y|^2/(4e^v)$ and the definition of the Gamma function
\cite[Appendix~A.2]{Grafakos} give
\[
 \int_{\mathbb{R}} e^{-2\pi i\xi v}h_{e^v}(x-y)\,d v
 =\pi^{-n/2}4^{2\pi i\xi}\Gamma\!\left(\frac n2+2\pi i\xi\right)
 |x-y|^{-n-4\pi i\xi}.
\]
An application of $-\Delta_x$ and the identity
\[
 \Delta_x|x-y|^{-\beta}=\beta(\beta+2-n)|x-y|^{-\beta-2}
\]
with $\beta=n+4\pi i\xi$ proves \eqref{eq:5.1}.  Absolute convergence of the final
integral follows from $P(x)<\infty$.  The kernel $|x-y|^{-n-4\pi i\xi}$ is
that of a Riesz potential of complex order; cf.\ \cite[Chapter~V,
Section~1]{SteinSI}.
\end{proof}

\begin{proof}[Proof of Theorem \ref{thm:exterior}]
We first establish the validity of the inequality \eqref{eq:frequency-tail-intro}.  Euler's limit formula
\[
\Gamma(z)=\lim_{m\to\infty} \frac{m^z\,m!}{\bigl(z(z+1)\cdots(z+m)\bigr)}
\]
\cite[Appendix~A.7]{Grafakos}, applied to $z=n/2+2\pi i\xi$ and to
$z=n/2$, gives the exact identity
\begin{align}
 \frac{|\Gamma(n/2+2\pi i\xi)|^2}{\Gamma(n/2)^2}
 =\prod_{k=0}^\infty
 \left(1+\frac{(2\pi\xi)^2}{(n/2+k)^2}\right)^{-1}. \label{eq:5.4}
\end{align}
We next establish the bound
\begin{align}
 \frac{|\Gamma(n/2+2\pi i\xi)|^2}{\Gamma(n/2)^2}
 \lesssim \exp\left[-c\min\left\{\frac{4\pi^2\xi^2}{n},2\pi|\xi|\right\}\right].
 \label{eq:5.5}
\end{align}
To this end, take the logarithm of equation \eqref{eq:5.4}.  If
$2\pi|\xi|\le n/2$, then $\log(1+x)\ge x/2$ for $0\le x\le1$ and
\[
\sum_{k\ge0}(n/2+k)^{-2}\ge2/n.
\]
For the case $2\pi|\xi|>n/2$, we use the fact that the terms multiplied are less than one, and in particular their logarithms are negative, to estimate the sum from above by dropping all but the 
$\asymp 2\pi|\xi|$ values of $k$ for which
\[2\pi|\xi|\le n/2+k\le4\pi|\xi|.\]
As each of these retained summands is bounded above by an absolute negative constant, together with the preceding argument for the small values of $|\xi|$, this completes the proof of \eqref{eq:5.5}.

The identity $|S^{n-1}|=n\omega_n=2\pi^{n/2}/\Gamma(n/2)$, integration of \eqref{eq:5.2}
over ${\Omega^c}$, and Proposition \ref{finite_P} give
\begin{align}
 \int_{\Omega^c}|\widehat F_{\sigma,x}(\xi)|\,d x
 \lesssim  C^n
 \frac{|\Gamma(n/2+2\pi i\xi)|}{\Gamma(n/2)}
 (n+2\pi|\xi|)(1+2\pi|\xi|)\mathfrak{m}.          \label{eq:5.6}
\end{align}
Note we kept the normalization in Proposition \ref{finite_P} for precise cancellation in this step:
\[
 \pi^{-n/2}\Gamma(n/2)\frac{n\omega_n}2=1.
\]
Fix $A\ge1$, to be chosen, and set $\Lambda=An$.  On the range
$|\xi|\ge An$ one has
\begin{align*}
 \min\{4\pi^2\xi^2/n,2\pi|\xi|\}&=2\pi|\xi|,\\
 (n+2\pi|\xi|)(1+2\pi|\xi|)&\lesssim \xi^2,
\end{align*}
so that using \eqref{eq:5.5} and integrating \eqref{eq:5.6} over $|\xi| \geq A n$ we obtain
\[
 \int_{|\xi|>An} \int_{\Omega^c} |\widehat F_{\sigma,x}(\xi)|\,d x d\xi
 \lesssim  \mathfrak{m} C^n  \int_{|\xi|>An} \xi^2e^{-c|\xi|}\, d\xi.
\]
Integration by parts gives
\[
 C^n \int_{An}^\infty\xi^2e^{-c\xi}\,d\xi
 =C^n e^{-cAn}\left(\frac{(An)^2}c+\frac{2An}{c^2}+\frac2{c^3}\right) \lesssim 1,
\]
for a sufficiently large absolute constant $A$ which depends on $c,C$ but is independent of $n$.  The preceding estimates and Tonelli's theorem finally yield
\begin{equation}
 \int_{\Omega^c}\int_{|\xi|>An}|\widehat F_{\sigma,x}(\xi)|
 \,d\xi\,d x
 \lesssim \mathfrak{m}.                                      \label{eq:5.8}
\end{equation}

We next establish the validity of \eqref{eq:frequency-prefix-intro}. For $\eta = \nu, \mu, \sigma$ we write, for convenience of notation,
\[
\mathcal E_\eta(B)
 =\int_{\Omega^c}\int_{|\xi|\le B}
 (2\pi\xi)^2|\widehat F_{\eta,x}(\xi)|^2\,d\xi\,d x.
\]
We first verify that, for $\eta = \nu,\mu,\sigma$ and almost every $x\in \Omega^c$, $F_{\eta,x} \in L^1(\mathbb{R})$, so that  $\widehat F_{\eta,x}$ is the usual Fourier transform of an integrable function.
First consider the atomic source $\nu$.  Since
\[
 \delta_\nu
 :=
 \operatorname{dist}(\operatorname{supp}\nu,\Omega^c)>0,
\]
for every $x\in\Omega^c$ and every $t>0$ we have
\[
 0\le H_t\nu(x)
 \le
 \mathfrak{m}(4\pi t)^{-n/2}
 \exp\left(-\frac{\delta_\nu^2}{4t}\right).
\]
The change of variables $t=e^v$ gives
\begin{align*}
 \int_{\mathbb{R}}|F_{\nu,x}(v)|\,dv
 &=
 \int_0^\infty H_t\nu(x)\,\frac{dt}{t} \\
 &\le
 \mathfrak{m}
 \int_0^\infty
 (4\pi t)^{-n/2}
 \exp\left(-\frac{\delta_\nu^2}{4t}\right)
 \frac{dt}{t} \\
 &=
 \frac{\Gamma(n/2)}{\pi^{n/2}}
 \frac{\mathfrak{m}}{\delta_\nu^n}
 <\infty.
\end{align*}
Thus $F_{\nu,x}\in L^1(\mathbb{R})$ for every
$x\in\Omega^c$, and $\widehat F_{\nu,x}$ is the Fourier
transform of an integrable function.

We next consider $\sigma=-\Delta u$.  Since $u$ is compactly supported
and $x\in\Omega^c$, integration by parts gives
\begin{align*}
 H_t\sigma(x)
 =
 -\int_\Omega u(y)\Delta h_t(x-y)\,dy.
\end{align*}
The heat kernel satisfies
\begin{align*}
 \Delta h_t(z)
 =
 \left(
 \frac{|z|^2}{4t^2}-\frac{n}{2t}
 \right)h_t(z).
\end{align*}
A change of variables $s=|z|^2/(4t)$ therefore gives
\begin{align*}
 \int_0^\infty
 |\Delta h_t(z)|\,\frac{dt}{t}
 \le
 C_n|z|^{-n-2},
 \qquad z\ne0.
\end{align*}
It follows from Tonelli's theorem that
\begin{align*}
 \int_{\mathbb{R}}|F_{\sigma,x}(v)|\,dv
 &=
 \int_0^\infty|H_t\sigma(x)|\,\frac{dt}{t} \\
 &\le
 \int_\Omega u(y)
 \int_0^\infty
 |\Delta h_t(x-y)|\,\frac{dt}{t}\,dy \\
 &\le
 C_n\int_\Omega
 \frac{u(y)}{|x-y|^{n+2}}\,dy
 =
 C_nP(x).
\end{align*}
By Proposition~\ref{finite_P}, $P(x)<\infty$ for almost every
$x\in\Omega^c$.  Consequently, $ F_{\sigma,x}\in L^1(\mathbb{R})$ for almost every $x\in\Omega^c$, and $\widehat F_{\sigma,x}$ is also a classical Fourier transform.

Since
$
 \sigma=\nu-\mu,
$
we have
\[
 F_{\mu,x}=F_{\nu,x}-F_{\sigma,x}.
\]
Therefore, $ F_{\mu,x}\in L^1(\mathbb{R})$ for almost every $x\in\Omega^c$, and linearity gives
\begin{equation*}
 \widehat F_{\mu,x}
 =
 \widehat F_{\nu,x}-\widehat F_{\sigma,x}.
\end{equation*}
Hence, for every $B\ge1$, we have the estimate
\begin{align} 
\int_{\Omega^c}\int_{|\xi|\le B}
 (2\pi\xi)^2
 |\widehat F_{\sigma,x}(\xi)|^2
 \,d\xi\,dx \le
 2\mathcal E_\nu(B)
 +
 2\int_{\Omega^c}\int_{|\xi|\le B}
 (2\pi\xi)^2
 |\widehat F_{\mu,x}(\xi)|^2
 \,d\xi\,dx, \label{key_intermediate}
\end{align}
where the Fourier transforms are standard Fourier transforms of integrable functions. 

In order to estimate the second term, we enlarge the integral to all of space and compute the Fourier transform of the distribution
\[
 \partial_vF_{\mu,x}(v)
 =
 e^v\partial_tH_t\mu(x)\big|_{t=e^v}.
\]
In particular, since $\mu\in L^2(\mathbb{R}^n)$, logarithmic-time Plancherel and the spectral theorem give
\begin{align*}
\int_{\mathbb{R}^n}\int_{\mathbb{R}}
 \left|\widehat  {\partial_vF_{\mu,x}} (\xi)\right|^2
 \,d\xi\,dx
 &=
 \int_{\mathbb{R}^n}\int_0^\infty
 |t\partial_tH_t\mu(x)|^2
 \frac{dt}{t}\,dx \\
 &=
 \frac14\|\mu\|_2^2
 =
 \frac14\kappa^2|\Omega|
 =
 \frac14\kappa\mathfrak{m}.
\end{align*}
For almost every $x\in\Omega^c$, the Fourier transform
identified above satisfies
\[
\widehat  {\partial_vF_{\mu,x}} (\xi)
 =
 2\pi i\xi\,\widehat F_{\mu,x}(\xi).
\]
Therefore, for every $B>0$,
\begin{align}
 \int_{\Omega^c}\int_{|\xi|\le B}
 (2\pi\xi)^2
 |\widehat F_{\mu,x}(\xi)|^2
 \,d\xi\,dx
 &\le \frac14\kappa\mathfrak{m}.
 \label{eq:mu-frequency-energy}
\end{align}
This estimates the second term on the right-hand side of \eqref{key_intermediate}.  

We next show how Theorem \ref{thm:resolvent-square} allows us to estimate the first term. For $x\in {\Omega^c}$, set $G_{N,x}(v)=A_{Ne^{-v}}^N\nu(x)$ and put
\[
 \Phi_N(r)=\frac{N^N}{\Gamma(N)}r^{N-1}e^{-Nr}
\]
for $r>0$.  %The function $\Phi_N$ is the Gamma density with shape and rate $N$. 
The integral formula for the Abel power gives
\begin{align*}
 G_{N,x}(v)
 =\int_0^\infty\Phi_N(r)F_{\nu,x}(v+\log r)\,d r.
\end{align*}
Taking the Fourier transform in $v$ and using the Gamma moment formula 
\begin{align*}
\int_0^\infty r^{N-1+2\pi i\xi}e^{-Nr}\,dr=N^{-N-2\pi i\xi}\Gamma(N+2\pi i\xi),
\end{align*}
see, e.g., \cite[Appendix~A.2]{Grafakos}, gives
\begin{align*}
 \widehat G_{N,x}(\xi)
 &=M_N(\xi)\widehat F_{\nu,x}(\xi),\\
 M_N(\xi)
 &=\int_0^\infty\Phi_N(r)r^{2\pi i\xi}\,d r
 =N^{-2\pi i\xi}\frac{\Gamma(N+2\pi i\xi)}{\Gamma(N)}.
\end{align*}
As noted above in
\eqref{eq:5.4}, Euler's limit formula \cite[Appendix~A.7]{Grafakos} gives
\[
 |M_N(\xi)|^2
 =\prod_{k=0}^\infty
 \left(1+\frac{(2\pi\xi)^2}{(N+k)^2}\right)^{-1}.
\]
Since $\sum_{k\ge0}(N+k)^{-2}\le2/N$, we have that
\[
 |M_N(\xi)|\ge e^{-1/4}
 \qquad\text{when}\qquad
 |\xi|\le\frac1{4\pi}\sqrt N.
\]
Because
$\partial_vG_{N,x}=-s\partial_s(A_s^N\nu)(x)$ with
$s=Ne^{-v}$, logarithmic-time Plancherel and
Theorem~\ref{thm:resolvent-square} give
\[
 \mathcal E_\nu\!\left(\frac1{4\pi}\sqrt N\right)
 \lesssim \sqrt N\,\kappa\mathfrak m.
\]
Given $B\ge1$, choose $N_B=\lceil16\pi^2B^2\rceil$.  Then
$B\le(4\pi)^{-1}\sqrt{N_B}$ and
$\sqrt{N_B}\le\sqrt{16\pi^2+1}\,B$, and so
\begin{equation*}
 \mathcal E_\nu(B)\lesssim B\,\kappa\mathfrak m.
\end{equation*}
The preceding inequality, in combination with \eqref{eq:mu-frequency-energy}, yields, for every $B\ge1$, the estimate
\begin{equation}\label{eq:frequency-profile-low}
\int_{\Omega^c}\int_{|\xi|\le B}
 (2\pi\xi)^2|\widehat F_{\sigma,x}(\xi)|^2\,d\xi\,d x \lesssim B\,\kappa\mathfrak m.
\end{equation}
The estimates \eqref{eq:5.8} and \eqref{eq:frequency-profile-low} complete the proof.
\end{proof}

\section{The \texorpdfstring{$L^2$}{L2} Estimate for \texorpdfstring{$T_\Lambda$}{T Lambda}} \label{lowfreq_log}
In this section we show how to convert \eqref{eq:frequency-profile-low} into an estimate for $T_\Lambda$ in $L^2(\Omega^c)$.

\begin{lemma}\label{lem:log-sobolev}
Let $\Lambda>1$ and suppose
\[
 \mathcal E_\sigma(B)
 =
 \int_{\Omega^c} \int_{|\xi|\le B}
 (2\pi\xi)^2|\widehat F_{\sigma,x}(\xi)|^2\,d\xi\,dx
 \le KB,
\]
for some $K>0$ and all $1\le B\le\Lambda$.  Then 
\[
 \int_{\Omega^c} T_\Lambda(x)^2\,dx
 \lesssim \log^2(e\Lambda)\,K .
\]
\end{lemma}

\begin{proof}
Set $J = \lceil \log_2 \Lambda \rceil$ and decompose the interval $\{|\xi|\le\Lambda\}$ into
\begin{align*}
 I_0=\{|\xi|\le1\} \qquad\text{and}\qquad
 I_j=\{2^{j-1}<|\xi|\le \min\{2^j,\Lambda\}\}
\end{align*}
for $j = 1,\ldots, J$. Define
\[
 T_j(x)
 :=
 \int_{I_j}
 |r(\xi)\widehat F_{\sigma,x}(\xi)|\,d\xi.
\]
Then
\[
 T_\Lambda(x)\le\sum_{j=0}^JT_j(x).
\]

On $I_0$, we use $|r(\xi)|\le2\pi|\xi|$ and apply Cauchy--Schwarz to obtain
\[
 \|T_0\|_{L^2(\Omega^c)}^2
 \lesssim
 \int_{\Omega^c}\int_{|\xi|\le1}
 (2\pi\xi)^2|\widehat F_{\sigma,x}(\xi)|^2\,d\xi\,dx
 \lesssim K.
\]
For $j\ge1$, we use $|r(\xi)|\le1$ to obtain
\[
 T_j(x)^2
 \le |I_j|\int_{I_j}|\widehat F_{\sigma,x}(\xi)|^2\,d\xi.
\]
Since $|I_j|\lesssim 2^j$ and $|\xi|\gtrsim 2^{j-1}$ on $I_j$,
\[
\begin{aligned}
 \|T_j\|_{L^2({\Omega^c})}^2
 \lesssim
 2^j\,2^{-2j}\mathcal E_\sigma(\min\{2^j,\Lambda\})\lesssim K.
\end{aligned}
\]
Using this constant bound for each of the intervals and the fact that $J \lesssim \log(e\Lambda)$, the triangle inequality gives
\[
 \|T_\Lambda\|_{L^2({\Omega^c})}
 \le\sum_{j=0}^J\|T_j\|_{L^2({\Omega^c})}
 \lesssim \log(e\Lambda)\sqrt K,
\]
which is the claimed estimate.
\end{proof}

\section{Proofs of the Main Results}\label{main}
\begin{proof}[Proof of Theorem \ref{thm:main}]

For dimensions $n=1,2$, the result follows from a majorization by the Hardy--Littlewood maximal function and any standard covering argument, thus we may restrict our attention to the case $n\ge3$.  By a well-known reduction \cite[Theorem 1]{MenarguezSoria1992}, it suffices to prove the inequality for any nonnegative atomic measure.  

Fix such a nonnegative atomic measure $\nu$ and $\alpha>0$.  Set $\Lambda=An$ for $A>1$ sufficiently large to satisfy the hypothesis of Theorem \ref{thm:exterior}.  Define the cap height of the balayage $\kappa$ by
\begin{equation*}
 L_n=\log(e\Lambda),
 \qquad
 \kappa=\frac{\alpha}{2L_n}.
\end{equation*}
An application of Proposition \ref{prop:balayage} with cap $\kappa$ yields corresponding $\mu$, $\sigma$, and $\Omega$.  As noted in the introduction, the identity \eqref{goodandbad} together with $H_t\mu\le\kappa$ and the non-positivity of the ergodic average of $\sigma$ yields the upper bound
\begin{equation}
 \mathcal{H}_{\ast}\nu(x)
 \le\kappa+T_\Lambda(x)+Q_\Lambda(x)
 \label{eq:6-pointwise-max}
\end{equation}
for almost every $x\in {\Omega^c}$, where $Q_\Lambda$ and $T_\Lambda$ are as defined in \eqref{high} and \eqref{low}, respectively.  Here we utilize the bound \eqref{R_bound}, 
\begin{align*}
\sup_{v \in \mathbb{R}} |R_x(v)| \leq  T_\Lambda(x)+Q_\Lambda(x),
\end{align*}
which we give a proof of here for the convenience of the reader.  

For almost every $x\in\Omega^c$, Theorem~\ref{thm:exterior} implies that
\[
Q_\Lambda(x)<\infty
\qquad\text{and}\qquad
\int_{|\xi|\leq\Lambda}(2\pi\xi)^2
|\widehat F_{\sigma,x}(\xi)|^2\,d\xi<\infty.
\]
For such an $x$, the Cauchy--Schwarz inequality gives
\begin{align*}
\int_{|\xi|\le\Lambda}
|r(\xi)\widehat F_{\sigma,x}(\xi)|\,d\xi
\le
\left(
\int_{|\xi|\le\Lambda}
\frac{|r(\xi)|^2}{(2\pi\xi)^2}\,d\xi
\right)^{1/2}
\left(
\int_{|\xi|\le\Lambda}
(2\pi\xi)^2|\widehat F_{\sigma,x}(\xi)|^2\,d\xi
\right)^{1/2}.
\end{align*}
Using $|r(\xi)|^2/(2\pi\xi)^2
=(1+(2\pi\xi)^2)^{-1}$ on $|\xi|\le\Lambda$
and $|r|\le1$ elsewhere, we obtain
\begin{align*}
\int_{\mathbb R}
|r(\xi)\widehat F_{\sigma,x}(\xi)|\,d\xi
\le \frac1{\sqrt2}
\left(
\int_{|\xi|\le\Lambda}
(2\pi\xi)^2|\widehat F_{\sigma,x}(\xi)|^2\,d\xi
\right)^{1/2}
+Q_\Lambda(x)<\infty.
\end{align*}
The substitution $t=e^{v-a}$ gives
\begin{align*}
\widetilde P_{e^v}\sigma(x)
=\int_0^\infty e^{-a}F_{\sigma,x}(v-a)\,da.
\end{align*}
This is convolution with
$k(a)=e^{-a}\mathbf1_{(0,\infty)}(a)$, and hence
\begin{align*}
R_x=F_{\sigma,x}-k*F_{\sigma,x}.
\end{align*}
Since $F_{\sigma,x}\in L^1(\mathbb R)$ for almost every
$x\in\Omega^c$ and $k\in L^1(\mathbb R)$, we also have
$R_x\in L^1(\mathbb R)$. Taking Fourier transforms gives
\begin{align*}
\widehat R_x(\xi)
=\left(1-\frac{1}{1+2\pi i\xi}\right)
\widehat F_{\sigma,x}(\xi)=r(\xi)\widehat F_{\sigma,x}(\xi).
\end{align*}
Fourier inversion and the continuity of $R_x$ now yield
\[
R_x(v)=\int_{\mathbb R}e^{2\pi i\xi v}
r(\xi)\widehat F_{\sigma,x}(\xi)\,d\xi
\]
for every $v\in\mathbb R$. Splitting this integral into
$|\xi|\leq\Lambda$ and $|\xi|>\Lambda$ proves
\[
\sup_{v\in\mathbb R}|R_x(v)|
\leq T_\Lambda(x)+Q_\Lambda(x),
\]
which is \eqref{R_bound}.

Thus, \eqref{R_bound} is justified and we continue with the estimate.  Because $L_n\ge1$, the first term on the right side of
\eqref{eq:6-pointwise-max} is at most $\alpha/2$.  Hence, by the standard $\alpha/2$ estimate and two different applications of Chebychev's inequality 
\begin{align*}
 |\{\mathcal{H}_{\ast}\nu>\alpha\}|
 &\le|\Omega|
 +\left|
 \left\{x\in {\Omega^c}\colon T_\Lambda(x)>\frac{\alpha}{4}\right\}
 \right|
 +\left|
 \left\{x\in {\Omega^c}\colon Q_\Lambda(x)>\frac{\alpha}{4}\right\}
 \right|\\
 &\le\frac{\mathfrak{m}}{\kappa}
 +\frac{16}{\alpha^2}\int_{\Omega^c}T_\Lambda(x)^2\,d x
 +\frac4\alpha\int_{\Omega^c}Q_\Lambda(x)\,d x.
\end{align*}
Equation \eqref{eq:frequency-tail-intro} in Theorem \ref{thm:exterior} asserts
\begin{equation}
 \int_{\Omega^c}Q_\Lambda(x)\,d x\lesssim \mathfrak m.
 \label{eq:6-high-part}
\end{equation}
Meanwhile, equation \eqref{eq:frequency-prefix-intro} in Theorem \ref{thm:exterior} and Lemma~\ref{lem:log-sobolev} give
\begin{equation}
 \int_{\Omega^c}T_\Lambda(x)^2\,d x
 \lesssim \log^2(e\Lambda)\,\kappa\mathfrak m.
 \label{eq:6-remainder-max}
\end{equation}
The estimates \eqref{eq:6-high-part} and \eqref{eq:6-remainder-max}, in combination with the preceding chain of inequalities for the superlevel set of the heat maximal function of $\nu$, together yield
\begin{align*}
  |\{\mathcal{H}_{\ast}\nu>\alpha\}| &\lesssim
 L_n\frac{\mathfrak{m}}{\alpha}
 +L_n^2\frac{\kappa\mathfrak{m}}{\alpha^2}
 +\frac{\mathfrak{m}}{\alpha} \lesssim L_n\frac{\mathfrak{m}}{\alpha}.
\end{align*}
As $\Lambda=An$ with $A>1$ absolute,
\[
 L_n\lesssim \log(1+n),
\]
which proves the estimate for positive atomic measures.

By the atomic reduction cited at the beginning of the proof, the
same estimate holds for every nonnegative
$f\in L^1(\mathbb R^n)$. For a general $f$, positivity of the heat
semigroup gives
\[
%|H_tf|\leq H_t|f|,
%\qquad
\mathcal H_{\ast}f\leq\mathcal H_{\ast}|f|.
\]
An application of the nonnegative estimate to $|f|$ completes the proof.
\end{proof}

\begin{proof}[Proof of Theorem~\ref{cor:hl-main}]

We begin by proving \eqref{prop:ball-heat-intro}. Following the argument in \cite{SteinStromberg}, by dilation it suffices to find a suitable choice of $t>0$ such that
\begin{align*}
\frac{\mathbf{1}_{B(0,1)}(x)}{\omega_n}  \lesssim \sqrt n \frac{1}{(4\pi t)^{n/2}} e^{-|x|^2/4t},
\end{align*}
which by monotonicity properties is true whenever it holds for $|x|=1$, i.e.,
\begin{align*}
\frac{1}{\omega_n} = \frac{\Gamma\left(\frac{n}{2}+1\right)}{\pi^{n/2}} \lesssim \sqrt n \frac{1}{(4\pi t)^{n/2}} e^{-1/4t}.
\end{align*}
The asymptotics of the Euler Gamma function \cite[Appendix~A.6]{Grafakos} give
\begin{align*}
\Gamma\left(\frac{n}{2}+1\right) \approx \sqrt{2\pi (n/2)} \left(\frac{n}{2e}\right)^{n/2}
\end{align*}
for the left-hand side, while the choice $t=1/(2n)$ gives 
\begin{align*}
 \frac{1}{(4\pi t)^{n/2}} e^{-1/4t} =  \frac{(n/2)^{n/2}}{\pi^{n/2}} e^{-n/2}
\end{align*}
for the right-hand side, and the claim is demonstrated.  Indeed, rescaling gives $\mathbf{1}_{B(x,r)}(y)/(\omega_n r^n)\lesssim \sqrt n\, h_{r
^2/(2n)}(x-y)$ for every $r>0$; integrating against $f\ge0$ and taking a supremum in $r$ gives \eqref{prop:ball-heat-intro}.\looseness=-1

Finally, Theorem~\ref{cor:hl-main} follows readily from Theorem \ref{thm:main} and the pointwise bound \eqref{prop:ball-heat-intro}:
\begin{equation*}
\| \mathcal{M}f\|_{L^{1,\infty}(\mathbb{R}^n)}  \lesssim \sqrt n \| \mathcal{H}_{\ast}|f|\|_{L^{1,\infty}(\mathbb{R}^n)}  \lesssim \sqrt n \log(1+n)  \|f\|_{L^1(\mathbb{R}^n)}
\end{equation*}
for any $f \in L^1(\mathbb{R}^n)$. 
\end{proof}

\section*{Artificial Intelligence Statement}
After the success of GPT in connecting partial balayage and the dimension-free weak-type $(1,1)$ bound for the Riesz transform in our previous paper \cite{OuyangSpectorStockdale}, the first author employed a Sol agent, Danus \cite{LGSWLJJCCZ2026}, and Rethlas \cite{JGJWSLCWW2026} to work on the dimensional bound for the Hardy--Littlewood maximal function, and did not invite Polya this time.  The observation made in equation \eqref{prop:ball-heat-intro} is one the authors had noted years ago when working on the problem: that an improvement to the estimate for the heat maximal operator that avoids ergodic averaging would result in an improvement to the estimate for the Hardy--Littlewood maximal function.  Despite various computations, the authors had not made progress in this direction.  

The proof of the logarithmic dependence of the constant for the heat maximal operator was developed by large language models (LLMs), through a combination of ChatGPT (GPT-5.6 Sol), Codex CLI and web interface, and mathematical reasoning agents.  The first breakthrough in the proof was GPT's observation that the dyadic heat maximal operator admits logarithmic-in-dimension bounds, but a pointwise comparison with the continuous operator necessarily incurs an exponential dimensional loss. %that the continuous maximal operator could not be controlled by its dyadic variant.  
The authors suggested that GPT bound the heat maximal function by its dyadic variant plus what is essentially a supremal operator over pieces of the Littlewood--Paley $g$-function, though there do not seem to be good bounds for the latter. However, as the former has relatively good dimensional dependence, the authors suggested a weighted upper bound that penalizes the dyadic variant and makes the other maximal operator smaller.  GPT seems to have interpreted this as the square function of the Abel power of order $N$.  This interpretation, combined with the authors' balayage suggestion and what it terms elementary spectral calculus computations, allowed the problem to be reduced to algebraic manipulation of the operators' symbols.  A first approximation returned with dimensional dependence $n^{1/3}\log^{1/3} (n)$ for the heat maximal operator, and then through several further interactions, the authors suggested that the dependence could be improved.  GPT steadily refined the computations for the Abel square function estimate, and finally returned a proof with logarithmic dependence that is somewhat more complicated (from our perspective) than the one presented.  Through further encouragement to simplify and, finally, the authors' work, the present version was produced.  In particular, the authors wrote the introduction, rewrote the material in a way we feel is more natural for humans to think about the question, conducted a literature review, and revised the presentation of the proofs.  

The authors have independently verified, validated, and rewritten all parts of the paper that were influenced by LLM-generated material, and they take full responsibility for the mathematical content of the paper.

\section*{Acknowledgments}

The authors thank Yuyuan Ouyang for helpful discussions regarding effective AI usage and Riju Basak for comments on a preliminary version of the manuscript.


\begin{bibdiv}
\begin{biblist}
\bib{AldazCubes}{article}{
  author={Aldaz, J. M.},
  title={The weak type $(1,1)$ bounds for the maximal function
    associated to cubes grow to infinity with the dimension},
  journal={Ann. of Math. (2)},
  date={2011},
  volume={173},
  number={2},
  pages={1013\ndash 1023},
  doi={10.4007/annals.2011.173.2.10},
}

\bib{Aldaz}{article}{
   author={Aldaz, J. M.},
   title={The Stein-Str\"omberg covering theorem in metric spaces},
   journal={J. Math. Anal. Appl.},
   volume={449},
   date={2017},
   number={2},
   pages={1741--1753},
   issn={0022-247X},
   review={\MR{3601614}},
   doi={10.1016/j.jmaa.2016.11.070},
}

\bib{AldazPerezLazaro}{article}{
  author={Aldaz, J. M.},
  author={P\'erez L\'azaro, F. J.},
  title={The best constant for the centered maximal operator on radial
    decreasing functions},
  journal={Math. Inequal. Appl.},
  volume={14},
  date={2011},
  number={1},
  pages={173\ndash 179},
}

\bib{Aubrun2009}{article}{
  author={Aubrun, G.},
  title={Maximal inequality for high-dimensional cubes},
  journal={Confluentes Math.},
  volume={1},
  date={2009},
  pages={169\ndash 179},
}



\bib{Bourgain1986AJM}{article}{
  author={Bourgain, J.},
  title={On high-dimensional maximal functions associated to convex bodies},
  journal={Amer. J. Math.},
  volume={108},
  date={1986},
  pages={1467\ndash 1476},
}

 
\bib{Bourgain1986Israel}{article}{
  author={Bourgain, J.},
  title={On the $L^p$-bounds for maximal functions associated to convex
    bodies in $\mathbb{R}^n$},
  journal={Israel J. Math.},
  volume={54},
  date={1986},
  pages={257\ndash 265},
}

 
 
 \bib{Bourgain2014Cube}{article}{
  author={Bourgain, J.},
  title={On the Hardy--Littlewood maximal function for the cube},
  journal={Israel J. Math.},
  volume={203},
  date={2014},
  pages={275\ndash 293},
}

\bib{BMSWCubes}{article}{
  author={Bourgain, J.},
  author={Mirek, M.},
  author={Stein, E. M.},
  author={Wr\'obel, B.},
  title={Dimension-free estimates for discrete Hardy--Littlewood averaging
    operators over the cubes in $\mathbb{Z}^d$},
  journal={Amer. J. Math.},
  volume={141},
  date={2019},
  number={3},
  pages={857\ndash 905},
}

\bib{BMSWBalls}{article}{
  author={Bourgain, J.},
  author={Mirek, M.},
  author={Stein, E. M.},
  author={Wr\'obel, B.},
  title={On discrete Hardy--Littlewood maximal functions over the balls in
    $\mathbb{Z}^d$: dimension-free estimates},
  book={
    title={Geometric Aspects of Functional Analysis. Vol. I},
    series={Lecture Notes in Math.},
    volume={2256},
    publisher={Springer, Cham},
  },
  date={2020},
  pages={127\ndash 169},
  doi={10.1007/978-3-030-36020-7\_8},
}


 
\bib{BMSWVariational}{article}{
  author={Bourgain, J.},
  author={Mirek, M.},
  author={Stein, E. M.},
  author={Wr\'obel, B.},
  title={Dimension-free variational estimates on $L^p(\mathbb{R}^d)$ for
    symmetric convex bodies},
  journal={Geom. Funct. Anal.},
  volume={28},
  date={2018},
  number={1},
  pages={58\ndash 99},
}
 
\bib{BMSWSurvey}{article}{
  author={Bourgain, J.},
  author={Mirek, M.},
  author={Stein, E. M.},
  author={Wr\'obel, B.},
  title={On the Hardy--Littlewood maximal functions in high dimensions:
    continuous and discrete perspective},
  book={
    title={Geometric Aspects of Harmonic Analysis},
    series={Springer INdAM Ser.},
    volume={45},
    publisher={Springer, Cham},
  },
  date={2021},
  pages={107\ndash 148},
}

\bib{Caffarelli1998}{article}{
  author={Caffarelli, L. A.},
  title={The obstacle problem revisited},
  journal={J. Fourier Anal. Appl.},
  volume={4},
  date={1998},
  number={4-5},
  pages={383\ndash 402},
  doi={10.1007/BF02498216},
}

\bib{Carbery1986}{article}{
  author={Carbery, A.},
  title={An almost-orthogonality principle with applications to maximal
    functions associated to convex bodies},
  journal={Bull. Amer. Math. Soc. (N.S.)},
  volume={14},
  date={1986},
  number={2},
  pages={269\ndash 273},
}


\bib{DeleavalGuedonMaurey}{article}{
  author={Deleaval, L.},
  author={Gu\'edon, O.},
  author={Maurey, B.},
  title={Dimension free bounds for the Hardy--Littlewood maximal operator
    associated to convex sets},
  journal={Ann. Fac. Sci. Toulouse Math. (6)},
  volume={27},
  date={2018},
  number={1},
  pages={1\ndash 198},
  doi={10.5802/afst.1567},
}


\bib{DK}{article}{
   author={Deleaval, L.},
   author={Kriegler, C.},
   title={Dimension free bounds for the vector-valued Hardy--Littlewood
   maximal operator},
   journal={Rev. Mat. Iberoam.},
   volume={35},
   date={2019},
   number={1},
   pages={101--123},
   issn={0213-2230},
   review={\MR{3914541}},
   doi={10.4171/rmi/1050},
}


\bib{DunfordSchwartz}{book}{
  author={Dunford, N.},
  author={Schwartz, J. T.},
  title={Linear Operators. Part I: General Theory},
  series={Pure and Applied Mathematics},
  volume={7},
  note={With the assistance of W. G. Bade and R. G. Bartle;
    reprinted in the Wiley Classics Library, John Wiley \& Sons,
    New York, 1988},
  publisher={Interscience Publishers},
  place={New York},
  date={1958},
}

\bib{Evans}{book}{
  author={Evans, L. C.},
  title={Partial Differential Equations},
  edition={2},
  series={Graduate Studies in Mathematics},
  volume={19},
  publisher={American Mathematical Society},
  place={Providence, RI},
  date={2010},
}

\bib{GG}{article}{
   author={Ganguly, P.},
   author={Ghosh, A.},
   title={Dimension free estimates for the vector-valued Hardy--Littlewood
   maximal function on the Heisenberg group},
   journal={J. Funct. Anal.},
   volume={290},
   date={2026},
   number={5},
   pages={Paper No. 111285, 26 pp.},
   issn={0022-1236},
   review={\MR{4993628}},
   doi={10.1016/j.jfa.2025.111285},
}



\bib{GardinerSjodin}{article}{
  author={Gardiner, S. J.},
  author={Sj\"odin, T.},
  title={Partial balayage and the exterior inverse problem of potential theory},
  conference={
    title={Potential Theory and Stochastics in Albac},
  },
  book={
    series={Theta Ser. Adv. Math.},
    volume={11},
    publisher={Theta},
    place={Bucharest},
  },
  date={2009},
  pages={111\ndash 123},
}

\bib{Grafakos}{book}{
  author={Grafakos, L.},
  title={Classical Fourier Analysis},
  edition={3},
  series={Graduate Texts in Mathematics},
  volume={249},
  publisher={Springer},
  place={New York},
  date={2014},
}

\bib{GustafssonRoos}{article}{
  author={Gustafsson, B.},
  author={Roos, J.},
  title={Partial balayage on Riemannian manifolds},
  journal={J. Math. Pures Appl. (9)},
  date={2018},
  volume={118},
  pages={82\ndash 127},
  doi={10.1016/j.matpur.2017.07.013},
}


\bib{GustafssonSakai}{article}{
  author={Gustafsson, B.},
  author={Sakai, M.},
  title={Properties of some balayage operators, with applications to quadrature domains and moving boundary problems},
  journal={Nonlinear Anal.},
  volume={22},
  date={1994},
  number={10},
  pages={1221\ndash 1245},
  doi={10.1016/0362-546X(94)90107-4},
}


\bib{HKM}{book}{
  author={Heinonen, J.},
  author={Kilpel\"ainen, T.},
  author={Martio, O.},
  title={Nonlinear Potential Theory of Degenerate Elliptic Equations},
  publisher={Dover Publications},
  place={Mineola, NY},
  date={2006},
  note={Unabridged republication of the 1993 original},
}


\bib{IakovlevStromberg}{article}{
  author={Iakovlev, A. S.},
  author={Str\"omberg, J.-O.},
  title={Lower bounds for the weak type $(1,1)$ estimate for the
    maximal function associated to cubes in high dimensions},
  journal={Math. Res. Lett.},
  date={2013},
  volume={20},
  number={5},
  pages={907\ndash 918},
  doi={10.4310/MRL.2013.v20.n5.a7},
}



\bib{JGJWSLCWW2026}{article}{
  author={Ju, H.},
  author={Gao, G.},
  author={Jiang, J.},
  author={Wu, B.},
  author={Sun, Z.},
  author={Chen, L.},
  author={Wang, Y.},
  author={Wang, Y.},
  author={Wang, Z.},
  author={He, W.},
  author={Wu, P.},
  author={Xiao, L.},
  author={Liu, R.},
  author={Dai, B.},
  author={Dong, B.},
  title={Automated conjecture resolution with formal verification},
  date={2026},
  eprint={arXiv:2604.03789},
}




\bib{KinderlehrerStampacchia}{book}{
  author={Kinderlehrer, D.},
  author={Stampacchia, G.},
  title={An Introduction to Variational Inequalities and Their
    Applications},
  series={Classics in Applied Mathematics},
  volume={31},
  publisher={Society for Industrial and Applied Mathematics},
  place={Philadelphia, PA},
  date={2000},
}


\bib{Komatsu}{article}{
  author={Komatsu, H.},
  title={Fractional powers of operators},
  journal={Pacific J. Math.},
  volume={19},
  date={1966},
  pages={285\ndash 346},
}

\bib{KMPW2023}{article}{
  author={Kosz, D.},
  author={Mirek, M.},
  author={Plewa, P.},
  author={Wr\'obel, B.},
  title={Some remarks on dimension-free estimates for the discrete
    Hardy--Littlewood maximal functions},
  journal={Israel J. Math.},
  volume={254},
  date={2023},
  pages={1\ndash 38},
}


\bib{LQ}{article}{
   author={Li, H.-Q.},
   author={Qian, B.},
   title={Centered Hardy--Littlewood maximal functions on Heisenberg type
   groups},
   journal={Trans. Amer. Math. Soc.},
   volume={366},
   date={2014},
   number={3},
   pages={1497--1524},
   issn={0002-9947},
   review={\MR{3145740}},
   doi={10.1090/S0002-9947-2013-05965-X},
}

\bib{LiebLoss}{book}{
  author={Lieb, E. H.},
  author={Loss, M.},
  title={Analysis},
  edition={2},
  series={Graduate Studies in Mathematics},
  volume={14},
  publisher={American Mathematical Society},
  place={Providence, RI},
  date={2001},
}

\bib{LGSWLJJCCZ2026}{article}{
  author={Liu, J.},
  author={Gao, G.},
  author={Sun, Z.},
  author={Wu, B.},
  author={Liu, S.},
  author={Jiang, J.},
  author={Ju, H.},
  author={Chen, L.},
  author={Cheng, R.},
  author={Zhang, X.},
  author={Dong, B.},
  title={Danus: orchestrating mathematical reasoning agents with fact-graph memory},
  date={2026},
  eprint={arXiv:2607.06447},
}

\bib{MenarguezSoria1992}{article}{
  author={Men{\'a}rguez, M. T.},
  author={Soria, F.},
  title={Weak type {$(1,1)$} inequalities of maximal convolution operators},
  journal={Rend. Circ. Mat. Palermo (2)},
  volume={41},
  date={1992},
  number={3},
  pages={342--352},
  doi={10.1007/BF02848939},
}

\bib{MSW}{article}{
   author={Mirek, M.},
   author={Szarek, T. Z.},
   author={Wr\'obel, B.},
   title={Dimension-free estimates for the discrete spherical maximal
   functions},
   journal={Int. Math. Res. Not. IMRN},
   date={2024},
   number={2},
   pages={901--963},
   issn={1073-7928},
   review={\MR{4692363}},
   doi={10.1093/imrn/rnac329},
}




\bib{Muller1990}{article}{
  author={M\"uller, D.},
  title={A geometric bound for maximal functions associated to convex
    bodies},
  journal={Pacific J. Math.},
  volume={142},
  date={1990},
  pages={297\ndash 312},
}

 

\bib{NaorTao2010}{article}{
  author={Naor, A.},
  author={Tao, T.},
  title={Random martingales and localization of maximal inequalities},
  journal={J. Funct. Anal.},
  volume={259},
  date={2010},
  pages={731\ndash 779},
}

\bib{NiWr}{article}{
   author={Niksi\'nski, J.},
   author={Wr\'obel, B.},
   title={Dimension-free estimates for discrete maximal functions and
   lattice points in high-dimensional spheres and balls with small radii},
   language={English, with English and French summaries},
   journal={J. Math. Pures Appl. (9)},
   volume={214},
   date={2026},
   pages={Paper No. 103955, 60 pp.},
   issn={0021-7824},
   review={\MR{5090433}},
   doi={10.1016/j.matpur.2026.103955},
}

\bib{Ouhabaz}{book}{
  author={Ouhabaz, E. M.},
  title={Analysis of Heat Equations on Domains},
  series={London Mathematical Society Monographs Series},
  volume={31},
  publisher={Princeton University Press},
  place={Princeton, NJ},
  date={2005},
}

\bib{OuyangSpectorStockdale}{article}{
  author={Ouyang, Y.},
  author={Spector, D.},
  author={Stockdale, C. B.},
  title={A dimension-free weak-type $(1,1)$ bound for the vector
    Riesz transform on $\mathbb R^n$},
  date={2026},
  eprint={arXiv:2608.18068},
}

\bib{PetrosyanShahgholianUraltseva}{book}{
  author={Petrosyan, A.},
  author={Shahgholian, H.},
  author={Uraltseva, N.},
  title={Regularity of Free Boundaries in Obstacle-Type Problems},
  series={Graduate Studies in Mathematics},
  volume={136},
  publisher={American Mathematical Society},
  place={Providence, RI},
  date={2012},
}


\bib{SteinSI}{book}{
  author={Stein, E. M.},
  title={Singular Integrals and Differentiability Properties of Functions},
  series={Princeton Mathematical Series},
  volume={30},
  publisher={Princeton University Press},
  place={Princeton, NJ},
  date={1970},
}

\bib{SteinTopics}{book}{
  author={Stein, E. M.},
  title={Topics in Harmonic Analysis Related to the
    Littlewood--Paley Theory},
  series={Annals of Mathematics Studies},
  volume={63},
  publisher={Princeton University Press},
  place={Princeton, NJ},
  date={1970},
}

\bib{Stein1982BAMS}{article}{
  author={Stein, E. M.},
  title={The development of square functions in the work of A. Zygmund},
  journal={Bull. Amer. Math. Soc. (N.S.)},
  volume={7},
  date={1982},
  pages={359\ndash 376},
}

\bib{Stein1983BAMS}{article}{
  author={Stein, E. M.},
  title={Some results in harmonic analysis in $\mathbb{R}^n$, for $n\to\infty$},
  journal={Bull. Amer. Math. Soc. (N.S.)},
  volume={9},
  date={1983},
  pages={71\ndash 73},
}
\bib{Stein1987ICM}{article}{
  author={Stein, E. M.},
  title={Problems in harmonic analysis related to curvature and
    oscillatory integrals},
  book={
    title={Proceedings of the International Congress of Mathematicians
      (Berkeley, Calif., 1986)},
    volume={1},
    publisher={Amer. Math. Soc.},
    place={Providence, RI},
  },
  date={1987},
  pages={196\ndash 221},
}

\bib{SteinStromberg}{article}{
  author={Stein, E. M.},
  author={Str\"omberg, J.-O.},
  title={Behavior of maximal functions in $\mathbb R^n$ for large $n$},
  journal={Ark. Mat.},
  date={1983},
  volume={21},
  pages={259\ndash 269},
  doi={10.1007/BF02384314},
}

\bib{NW}{article}{
   author={Wang, P.},
  author={Nie, X.},
   title={Dimension-free estimates for maximal functions with nonisotropic
   dilations},
   journal={Potential Anal.},
   volume={63},
   date={2025},
   number={1},
   pages={255--274},
   issn={0926-2601},
   review={\MR{4937885}},
   doi={10.1007/s11118-024-10170-4},
}


\bib{Z}{article}{
   author={Zienkiewicz, J.},
   title={Estimates for the Hardy--Littlewood maximal function on the
   Heisenberg group},
   journal={Colloq. Math.},
   volume={103},
   date={2005},
   number={2},
   pages={199--205},
   issn={0010-1354},
   review={\MR{2197849}},
   doi={10.4064/cm103-2-5},
}


\end{biblist}
\end{bibdiv}
\end{document}